\documentclass{amsart}
\usepackage{amsthm}
\usepackage{amsmath}
\usepackage{amstext}
\usepackage{amsfonts}
\usepackage{amssymb}
\usepackage{latexsym}
\usepackage{amscd}
\usepackage{xypic}
\usepackage{enumitem}
\theoremstyle{plain}
\newtheorem{theorem}{Theorem}[section]
\newtheorem{lemma}[theorem]{Lemma}
\newtheorem{proposition}[theorem]{Proposition}
\newtheorem{corollary}[theorem]{Corollary}

\theoremstyle{definition}

\newtheorem{example}[theorem]{Example}
\newtheorem{remark}[theorem]{Remark}

\numberwithin{equation}{theorem}
\newtheorem{claim}[theorem]{Claim}

\begin{document}

\title[Vector fields on canonically polarized surfaces.]{Vector fields on canonically polarized surfaces.}
\author{Nikolaos Tziolas}
\address{Department of Mathematics, University of Cyprus, P.O. Box 20537, Nicosia, 1678, Cyprus}
\email{tziolas@ucy.ac.cy}
\thanks{Part of this paper was written during the author's stay at the Max Planck Institute for Mathematics in Bonn, from 01.02.2019 to 31.07.2019.}


\dedicatory{To my little daughter Eleonora.}


\begin{abstract}
This paper investigates the geometry of  canonically polarized surfaces  defined over a field of positive characteristic which have a nontrivial global vector field, and the implications that the existence of such surfaces has in the moduli problem of canonically polarized surfaces.

In particular, an explicit integer valued function $f(x)$ is obtained with the following properties. If $X$ is a  canonically polarized surface defined over an algebraically closed field of characteristic $p>0$ such that 
$p>f(K_X^2) $ and $X$ has a nontrivial global vector field, then $X$ is unirational and the algebraic fundamental group is trivial. As a consequence of this result, large classes of canonically polarized surfaces are identified whose moduli stack is Deligne-Mumford, a property that does not hold in general in positive characteristic.
\end{abstract}

\maketitle

\section{Introduction}

The objective of this paper is to investigate the geometry of canonically polarized surfaces with nontrivial global vector fields and to use the results of this investigation in order to study the moduli stack of canonically polarized surfaces in positive characteristic. An investigation with these objectives was initiated in~\cite{Tz17} where the case of smooth canonically polarized surfaces $X$ with $K_X^2\leq 2$ has been studied.

A normal projective surface $X$ defined over an algebraically closed field is called canonically polarized if and only if $K_X$ is ample and $X$ has canonical singularities, or equivalently the singularities of $X$ are rational double points. Canonically polarized surfaces are precisely the canonical models of smooth minimal surfaces of general type and they play a fundamental role in the classification problem of surfaces of general type. In fact, early on in the theory of moduli of surfaces of general type, it was realized that the moduli functor of surfaces of general type is not well behaved and that the correct objects to parametrize are not the surfaces of general type but instead their canonical models~\cite{Ko10}, i.e., the canonically polarized surfaces.

The property that a  canonically polarized surface $X$ has a nontrivial global vector field is equivalent to the property that its automorphism scheme $\mathrm{Aut}(X)$ is not smooth. The reason is that the space of global vector fields of $X$ is canonically isomorphic to $\mathrm{Hom}(\Omega_X,\mathcal{O}_X)$, the tangent space at the identity of $\mathrm{Aut}(X)$. Moreover, it is well known that if $X$ is canonically polarized then $\mathrm{Aut}(X)$ is a zero dimensional scheme of finite type over the base field. Therefore the existence of nontrivial global vector fields on $X$ is equivalent to the non smoothness of $\mathrm{Aut}(X)$ and consequently the existence of non trivial infinitesimal automorphisms of $X$. Considering that $\mathrm{Aut}(X)$ is a group scheme and every group scheme in characteristic zero is smooth, non smoothness of $\mathrm{Aut}(X)$ can happen only in positive characteristic. Therefore a  canonically polarized surface can have non trivial global vector fields only when it is defined over a field of positive characteristic. 

Examples of smooth canonically polarized surfaces surfaces with nontrivial global vector fields exist but are hard to find since by~\cite[Lemma 4.1]{Tz17} such surfaces are not liftable to characteristic zero. Such examples have been found by H. Kurke~\cite{Ku81}, W. Lang~\cite{La83} and N. I. Shepherd-Barron~\cite{SB96}. Singular examples are much easier to find and in fact there exists many examples of canonically polarized surfaces with nontrivial global vector fields that are even liftable to characteristic zero. Such an example is given in   Example~\ref{example}.

The existence of nontrivial global vector fields on canonically polarized surfaces is intimately related to fundamental properties of the local and global moduli functors, in particular the moduli stack. 

From the local moduli point of view, suppose that $X$ is a canonically polarized surface defined over a field of characteristic $p$.  If $p=0$ then the local deformation functor  $Def(X)$ is pro-representable since in this case, as explained earlier, 
$\mathrm{Hom}(\Omega_X,\mathcal{O}_X)=0$ and hence $X$ has no infinitesimal deformations~\cite[Corollary 2.6.4]{Se06}. The pro-representability of $Def(X)$ implies the existence of a universal family for the local moduli functor, an ideal solution to the moduli problem. However, if $p>0$, $X$ may have nontrivial infinitesimal automorphisms due to the existence of nontrivial global vector fields and hence $Def(X)$ is not pro-representable but only has a hull.

From the global moduli point of view, it is well known~\cite{KSB88}~\cite{Ko97} that the moduli stack of canonically polarized surfaces is a separated  Artin stack of finite type over the base field with zero dimensional stabilizers. In characteristic zero the stack is in fact a Deligne-Mumford stack. This implies that there exists a family $\mathcal{X} \rightarrow S$ such that for any variety $X$ in the moduli problem, there exists finitely many $s\in S$ such that $\mathcal{X}_s\cong X$, up to \'etale base change any other family is obtained from it by base change and that for any closed point $s \in S$, the completion $\hat{\mathcal{O}}_{S,s}$ pro-represents the local deformation functor $Def(X_s)$. However, none of these hold in general in characteristic $p>0$.  The reason for this failure is the existence of canonically polarized surfaces with non smooth automorphism scheme, or equivalently with nontrivial global vector fields~\cite[Theorem 4.1]{DM69}. In some sense then the existence of nontrivial global vector fields on canonically polarized surfaces is the obstruction for the moduli stack to be Deligne-Mumford. 


This  investigation has two main objectives. 

The first objective is to find numerical conditions, which imply that the moduli stack of canonically polarized surfaces is Deligne-Mumford and the local deformation functor pro-representable. According to~\cite[Theorem 3.1]{Tz17} such conditions exist. However their existence is due to purely theoretical reasons and no explicit conditions were obtained so far.

The second  objective is  to describe the geometry of canonically polarized surfaces which have nontrivial global vector fields and consequently their moduli stack is not Deligne-Mumford. The hope is to obtain a good insight in the geometry of such surfaces that will allow the modification of the moduli problem in order to get a better moduli theory for these surfaces. 

From the existing examples of canonically polarized surfaces with nontrivial global vector fields and the case of smooth canonically polarized surfaces with $K^2\leq 2$, one gets the feeling that surfaces with nontrivial global vector fields tend to be uniruled and simply connected~\cite{Tz17}. However non uniruled examples exist in characteristic 2~\cite{SB96}, but it is unknown if non uniruled examples exist in higher characteristics. 

The main results of this paper are the following. 


\begin{theorem}\label{intro-the-1}
Let $X$ be a canonically polarized surface over an algebraically closed field of characteristic $p>0$. Suppose that $X$ has a nontrivial global vector field, or equivalently $\mathrm{Aut}(X)$ is not reduced and  that 
\[
p > \mathrm{max} \{8(K_X^2)^3+12(K_X^2)^2+3, 4508K_X^2+3\}.
\]
Then $X$ is unirational and $\pi_1(X)=\{1\}$.
\end{theorem}

The contrapositive of the previous theorem provides numerical condition between $K_X^2$, and $p$ which implies the reducedness of the automorphism $\mathrm{Aut}(X)$.


If the automorphism scheme $\mathrm{Aut}(X)$ of $X$ is not smooth then $\mathrm{Aut}(X)$ contains a subgroup scheme isomorphic to either $\alpha_p$ or $\mu_p$. This is equivalent to say that if $X$ has a nontrivial global vector field then $X$ has a nontrivial global vector field $D$ such that $D^p=0$ or $D^p=D$~\cite{Tz17b},~\cite{RS76}. If $\mu_p$ is a subgroup scheme of $\mathrm{Aut}(X)$, then finer restrictions can be imposed between $K_X^2$ and $p$  which imply the unirationality of $X$.

\begin{theorem}\label{intro-the-2}
Let $X$ be a canonically polarized surface over an algebraically closed field of characteristic $p>0$.  Suppose that  $\mu_p\subset \mathrm{Aut}(X)$, or equivalently that $X$ has a nontrivial vector field of multiplicative type and that one of the following happens:
\begin{enumerate}
\item $K_X^2=1$ and $p>211$.
\item $K_X^2\geq 2$ and $p >156K_X^2+3.$
\end{enumerate}
Then $X$ is unirational and $\pi_1(X)=\{1\}$.
\end{theorem}


The previous results have immediate applications to the structure of the local and global moduli problems of canonically polarized surfaces. 

\begin{theorem}\label{intro-the-4}
Let $X$ be a canonically polarized surface defined over an algebraically closed field of characteristic $p>0$. Suppose that $\pi_1(X)\not=\{1\}$ and that
\[
p > \mathrm{max} \{8(K_X^2)^3+12(K_X^2)^2+3, 4508K_X^2+3\}.
\]
Then $Def(X)$ is pro-representable.
\end{theorem}

\begin{theorem}\label{intro-the-3}
Let $k$ be a field of characteristic $p>0$ and $a\in \mathbb{Z}$  such that
\[
p > \mathrm{max} \{8a^3+12a^2+3, 4508a+3\}.
\]
Let $\mathcal{M}_{a}^{\mathrm{ntfg}}$ be the moduli stack of canonically polarized surfaces $X$ with $K_X^2= a$,   and nontrivial fundamental group. 
Then  $\mathcal{M}_{a}^{\mathrm{ntfg}}$ is Deligne-Mumford.
\end{theorem}
Theorem~\ref{intro-the-4} is an immediate consequence of Theorem~\ref{intro-the-1} and~\cite[Corollary 2.6.4]{Se06} while Theorem~\ref{intro-the-3} is a consequence of Theorem~\ref{intro-the-1} and~\cite[Theorem 4.1]{DM69} since the assumptions in both theorems imply that the automorphism scheme is reduced and that there exist no infinitesimal automorphisms. 

Taking into consideration the breadth of the possible values of the fundamental group of canonically polarized surfaces (it can be finite or infinite)~\cite{BCP11}, one sees that the previous results apply to a very large class of canonically polarized surfaces.

There are a few comments that I would like to make regarding the statement of Theorems~\ref{intro-the-1},~\ref{intro-the-2}.

The reason that the cases $K_X^2=1$ and $K_X^2\geq 2$ have been distinguished in Theorem~\ref{intro-the-2} is the following. In the proof of  Theorem~\ref{intro-the-2}, it is necessary to work with a base point free pluricanonical linear system $|mK_X|$. If $K_X^2=1$, then $|4K_X|$ is base point free while if $K_X^2\geq 2$, $|3K_X|$ is base point free~\cite{Ek88}. Otherwise the proofs are identical. One could work with $|4K_X|$ in both cases and get a unified statement but in this case the bounds obtained would be weaker.


The bounds on $K^2$ obtained in Theorem~\ref{intro-the-1} are not optimal if applied in specific cases. In particular, take the case when $K_X^2=1$. Then Theorem~\ref{intro-the-1} says that $X$ is unirational and simply connected if $p>4511$. However, if $X$ is smooth, it has been proved in~\cite{Tz17}, that $X$ is unirational and simply connected for all $p$ except possibly for $p= 3,5,7$. I believe that the methods developed in this paper to treat singular surfaces together with the techniques in~\cite{Tz17} will make it possible to obtain much finer bounds than those obtained in Theorem~\ref{intro-the-1} to the case of singular canonically polarized surfaces with $K_X^2=1$.

However, I believe that the strength of Theorem~\ref{intro-the-1} lies in its generality and not the optimality of the bounds obtained when applied in specific cases. The results apply to every canonically polarized surface and not to a specific class of them. In individual cases, like the cases when $K_X^2 \leq 2$ which have been treated in~\cite{Tz17} finer results might be obtained by exploiting known results about the geometry of the surfaces in question.

A desired result would be to obtain an inequality of the form $p>f(K_X^2)$,  which implies the smoothness of 
$\mathrm{Aut}(X)$. Such a result will make it possible to obtain a theorem like Theorem~\ref{intro-the-3} for canonically polarized surfaces whose fundamental group is not trivial as well. However, the bounds for $p$ are most likely going to be larger than those in Theorems~\ref{intro-the-1},~\ref{intro-the-2} making such a result weaker, since it would cover less cases, compared to Theorems~\ref{intro-the-1},~\ref{intro-the-2} for surfaces whose fundamental group is not trivial. I believe that a method based on the methods used in this paper should provide such a bound. However, at the moment I am unable to do so.

The reason that in Theorem~\ref{intro-the-2} I was able to obtain better bounds in the case when $X$ has a vector field of multiplicative type, or equivalently when $\mu_p$ is a subgroup scheme of $\mathrm{Aut}(X)$, is that $\mu_p$ is a diagonalizable group scheme while $\alpha_p$ is not. As a consequence of this there are many integral curves of the vector field on $X$, something that provides a lot of information about the geometry of $X$.

Finally I would like to say a few words about the proof of Theorems~\ref{intro-the-1},~\ref{intro-the-2}. The main idea of the proof is to show that under certain relations between $K_X^2$ and $p$, if $X$ has a nontrivial global vector field, then a linear system on $X$, usually of the form $|mK_X|$ contains a one dimensional subsystem $|V|$ consisting of integral curves of $D$. Then, to show that every irreducible component of every member of $|V|$ is a rational curve (usually singular) which will imply that either $X$ is birationally ruled (impossible in the case of canonically polarized surfaces) or more relations between $K_X^2$ and $p$. 
In the implementation of this strategy, it is necessary to find conditions under which the vector field fixes the singular points of $X$ and lifts to the minimal resolution of $X$, something, unlike in characteristic zero, is not always true in positive characteristic.

This paper is organized as follows.

In Section~\ref{sec-1} results about the number of singularities of a canonically polarized surface and conditions under which a vector field fixes the singularities of a surface and lifts to its minimal resolution are obtained. In particular, Theorem~\ref{sec11-th-1} provides un upper bound for the singular points of a canonically polarized surface $X$ as a function of $K_X^2$ and $\chi(\mathcal{O}_X)$. The result is under the assumption that the surface has a global vector field, an admittedly strong condition but sufficient for the purposes of this paper. In characteristic zero, similar bounds have been obtained by Y. Miyaoka~\cite{M84}. However, in my knowledge, no such results existed yet in positive characteristic. In Theorem~\ref{sec11-th-2}, similar conditions are obtained which imply that a vector field fixes the singular points and lifts to the minimal resolution. In characteristic zero this is always true but not in general in positive characteristic. This is exhibited in Example~\ref{example}.

In Section~\ref{sec-2} various results related to the geometry of integral curves of a vector field  on a surface are obtained which are needed in the proofs of Theorems~\ref{intro-the-1},~\ref{intro-the-2}.

In Section~\ref{sec-3} the general method and strategy for the proof of Theorems~\ref{intro-the-1},~\ref{intro-the-2} are explicitly described. 

Sections~\ref{sec4},~\ref{sec-5},~\ref{sec-6} are devoted to the proof of the main theorems. The statements  of Theorems~\ref{intro-the-1},~\ref{intro-the-2}  is the combination of the statements of Propositions~\ref{sec4-prop}~\ref{sec5-prop},~\ref{sec6-prop}.

\section{Notation-Terminology}
Let $X$ be an integral  scheme of finite type over an algebraically closed field $k$ of characteristic $p>0$.

Let $P\in X$ be a normal surface singularity and $f \colon Y \rightarrow X$ its minimal resolution. $P\in X$ is called a canonical singularity if and only if $K_Y=f^{\ast}K_X$. Two dimensional canonical singularities are precisely the rational double points (or Du Val singularities)  which are classified by explicit equations in all characteristics by M. Artin~\cite{Ar77}.

A normal projective surface $X$ is called a  canonically polarized surface if and only if $X$ has canonical singularities and $K_X$ is ample. These surfaces are exactly the canonical models of minimal surfaces of general type. 

$\mathrm{Der}_k(X)$ denotes the space of global $k$-derivations of $X$ (or equivalently of global vector fields). It is canonically identified with $\mathrm{Hom}_X(\Omega_X,\mathcal{O}_X)$.

Let $D$ be a nontrivial global vector field on $X$. $D$ is called $p$-closed if and only if $D^p=\lambda D$, for some $\lambda\in k$. $D$ is called of additive type if $D^p=0$ and of multiplicative type if $D^p=D$.
The fixed locus of $D$ is the closed subscheme of $X$ defined by the ideal sheaf $(D(\mathcal{O}_X))$. 
The divisorial part of the fixed locus of $D$ is called the divisorial part of $D$.  A point $P\in X$ is called an isolated singularity of $D$ if and only if the ideal of $\mathcal{O}_{X,P}$ generated by $D(\mathcal{O}_{X,P})$  has an associated prime of height $\geq 2$. 

A prime divisor $Z$ of $X$ is called an integral divisor of $D$ if and only if locally there is a derivation $D^{\prime}$ of $X$ such that $D=fD^{\prime}$, $f \in k(X)$,  $D^{\prime}(I_Z)\subset I_Z$ and $D^{\prime}(\mathcal{O}_X) \not\subset I_Z$ ~\cite{RS76}.

The vector field is said to stabilize a closed subscheme $Y$ of $X$ if and only if $D(I_Y) \subset I_Y$, where $I_Y$ is the ideal sheaf of $Y$ in $X$. If $Y$ is reduced and irreducible and is not contained in the divisorial part of $D$ then $Y$ is also an integral curve of $D$.

Let $X$ be a normal surface and $D$ a nontrivial global vector field on $X$ of either additive of multiplicative type. Then $D$ induces an $\alpha_p$ or $\mu_p$ action on $X$. Let $\pi \colon X\rightarrow Y$ be the quotient of $X$ by this action~\cite[Theorem 1, Page 104]{Mu70}. Let $C\subset X$ be a reduced and irreducible curve and $\tilde{C}=\pi(C)$. Suppose that $C$ is an integral curve of $D$. Then $\pi^{\ast}\tilde{C}=C$. Suppose that $C$ is not an integral curve of $D$. Then $\pi^{\ast}\tilde{C}=pC$~\cite{RS76}.

For any prime number $l\not= p$, the cohomology groups $H_{et}^i(X,\mathbb{Q}_l)$ are independent of $l$, they are finite dimensional of $\mathbb{Q}_l$ and are called the $l$-adic cohomology groups of $X$. The $i$-Betti number $b_i(X)$ of $X$ is defined to be the dimension of $H_{et}^i(X,\mathbb{Q}_l)$. It is well known that $b_i(X)=0$ for any $i>2n$, where $n=\dim X$~\cite[Chapter VI, Theorem 1.1]{Mi80}. 

$X$ is called  simply connected if $\pi_1(X)=\{1\}$, where $\pi_1(X)$ is the algebraic fundamental group of $X$.

Let $\mathcal{F}$ be a coherent sheaf on $X$. By $\mathcal{F}^{[n]}$ we denote the double dual $(\mathcal{F}^{\otimes n})^{\ast\ast}$.

\section{Singular points of surfaces with vector fields.}\label{sec-1}
Let $X$ be a normal projective surface defined over an algebraically closed field $k$ of characteristic $p>0$ whose singularities are rational double points. Suppose that $X$ has a nontrivial global vector field $D$. This section has two main oblectives. The first objective is  to obtain an upper bound, as a function of numerical invariants of $X$, of the number of singular points of $X$. The second objective is to find conditions which imply that the singular points of $X$ are fixed points of the vector field $D$ and that $D$ lifts to the minimal resolution of $X$.

If the base field has characteristic zero, then an upper bound of the number of singular points of $X$ was obtained by Y. Miyaoka~\cite{M84}. The proof of that result uses, among other characteristic zero techniques,  the Bogomolov-Miyaoka-Yau inequality which fails in positive characteristic. In this section, a result in that spirit is given under the assumption that $X$ has a nontrivial global vector field. This is a strong restriction on $X$, but it suffices for the purpose of this paper.

In characteristic zero, a vector field fixes the singularities and lifts to the minimal resolution~\cite{BW74}. However, this does not hold in general in positive characteristic. In fact something more interesting happens. 
There exist smooth minimal surfaces of general type without nontrivial global vector fields (and hence reduced automorphism scheme) whose canonical model has nontrivial global vector fields and therefore non reduced automorphism scheme. This is a situation that complicates the structure of the moduli of surfaces of general type in positive characteristic. The next example exhibits exactly such a case.

\begin{example}\label{example}
Let $k$ be an algebraically closed field of characteristic $2$ and $X \subset \mathbb{P}^3_k$ be the quintic given by
\[
x_1x_2(x_1^{3}+x_2^{3}+x_3^{3})+x_3^3x_4^2+x_3x_4^{4}=0.
\]
I will show the following:
\begin{enumerate}
\item The singularities of $X$ are rational double points of type $A_1$ ( i.e.,  locally isomorphic to $xy+z^2=0$) and $K_X$ is ample.
\item $X$ has nontrivial global vector fields and hence the automorphism scheme $\mathrm{Aut}(X)$ is a non reduced zero dimensional scheme. 
\item The vector fields of $X$ do not fix all the singular points of $X$ and therefore they do not lift to the minimal resolution of $X$.
\item The minimal resolution of $X$ is a smooth minimal surface of general type without vector fields and therefore with reduced automorphism scheme $\mathrm{Aut}(X^{\prime})$. 
\end{enumerate}
I proceed to show the above properties.

$X$ is a quintic in $\mathbb{P}^3$ and hence by the standard adjunction formula, $\mathcal{O}_X(K_X)=\mathcal{O}_X(1)$ and hence it is very ample.

The equation of $X$ is invariant under the graded derivation $D=x_3\frac{\partial}{\partial x_4}$ of $k[x_1,x_2,x_3,x_4]$, which therefore induces a nonzero global vector field on $X$ such that $D^2=0$. Hence $X$ has nontrivial global vector fields.

The singularities of $X$ can be checked locally. In the affine chart given by $x_3=1$, $X$ is given by the equation 
\[
x_1x_2+x_1^4x_2+x_1x_2^4+x_4^2+x_4^4=0
\]
in $k[x_1,x_2,x_4]$. The singular points of $X$ are those with $x_i^4+x_i=0$, $i=1,2$ and $x_4^4+x_4^2+x_1x_2=0$. A straightforward calculation shows that the degree two term of the polynomial  defining $X$ at every singular point is an irreducible quadric in $x_1$, $x_2$ and $x_4$ and hence the singularities of $X$ are ordinary double points given locally analytically by $xy+z^2=0$. Similarly one can easily check that there are no more singularities in the other charts. Hence (1) is proved.

In  this chart the vector field $D$ is given by  $D=\frac{\partial}{\partial x_4}$. Hence $D$ has no fixed points in the open set $x_3=1$. In particular,  none of the singular points is fixed by $D$.

Since $K_X$ is ample and $X$ has rational double points, $\mathrm{Aut}(X)$ is zero dimensional. Then since its tangent space is $\mathrm{Hom}(\Omega_X,\mathcal{O}_X)\not=0$, the space of global derivations, $\mathrm{Aut}(X)$ is not reduced. Hence (2) is proved.

Let now $f \colon X^{\prime}\rightarrow X$ be the minimal resolution of $X$. Then $X^{\prime}$ is simply the blow up of the singular points of $X$. Since $X$ has rational double points, $K_{X^{\prime}}=f^{\ast}K_X$ and therefore $X^{\prime}$ is a minimal surface of general type.

Since $f$ is the blow up of the singular points of $X$, a vector field on $X$ lifts to a vector field on $X^{\prime}$ if and only if it fixes the singular points of $X$. In addition, every vector field on $X^{\prime}$ induces a vector field on $X$ by the natural map $f_{\ast}T_{X^{\prime}} \rightarrow T_X$.  Therefore, in order to show that $X^{\prime}$ has no non trivial global vector fields, it suffices to show that there is no non trivial global vector field on $X$ which fixes every singular point of $X$. This will be done by explicitly calculating the vector fields of $X$.

\textbf{Claim:} A vector field on $X$ is the restriction on $X$ of a vector field on $\mathbb{P}^3$ which fixes $X$. 

Dualizing  the exact sequence
\[
0 \rightarrow \mathcal{O}_X(-5) \rightarrow \Omega_{\mathbb{P}^3} \otimes \mathcal{O}_X \rightarrow \Omega_X \rightarrow 0
\]
we get the exact sequence
\[
0 \rightarrow \mathrm{Hom}(\Omega_X,\mathcal{O}_X) \rightarrow \mathrm{Hom}(\Omega_{\mathbb{P}^3},\mathcal{O}_X) \rightarrow \mathrm{Hom}(\mathcal{O}_X(-5),\mathcal{O}_X).
\]
Moreover, there exists a natural exact sequence
\begin{gather*}
0 \rightarrow \mathrm{Hom}(\Omega_{\mathbb{P}^3}, \mathcal{O}_{\mathbb{P}^3}(-5) ) \rightarrow \mathrm{Hom}(\Omega_{\mathbb{P}^3},\mathcal{O}_{\mathbb{P}^3}) 
 \stackrel{\sigma}{\rightarrow} \mathrm{Hom}(\Omega_{\mathbb{P}^3},\mathcal{O}_{X}) \rightarrow \mathrm{Ext}^1((\Omega_{\mathbb{P}^3}, \mathcal{O}_{\mathbb{P}^3}(-5) ).
 \end{gather*}
Now $\mathrm{Ext}^1((\Omega_{\mathbb{P}^3}, \mathcal{O}_{\mathbb{P}^3}(-5) )=H^1(T_{\mathbb{P}^3}(-5))=0$, by using the standard exact sequence for the tangent sheaf on $\mathbb{P}^3$ and the cohomology of $\mathbb{P}^3$. Hence the map $\sigma$ is surjective and therefore every global vector field on $X$ is induced by a vector field on $\mathbb{P}^3$, and the claim is proved.

Now $h^0(T_{\mathbb{P}^3})=15$ and the global vector fields on $\mathbb{P}^3$ are induced by the following graded vector fields of $k[x_1,x_2,x_3,x_4]$. $D_1=x_1\frac{\partial}{\partial x_1}$, 
$D_2=x_2\frac{\partial}{\partial x_2}$, $D_3=x_4\frac{\partial}{\partial x_4}$, $D_4=x_1\frac{\partial}{\partial x_2}$, $D_5= x_1\frac{\partial}{\partial x_3}$, $D_6=x_1\frac{\partial}{\partial x_4}$, 
$D_7=x_2\frac{\partial}{\partial x_1}$, $D_8=x_2\frac{\partial}{\partial x_3}$, $D_9=x_2\frac{\partial}{\partial x_4}$, $D_{10}=x_3\frac{\partial}{\partial x_1}$, $D_{11}=x_3\frac{\partial}{\partial x_2}$, 
$D_{12}=x_3\frac{\partial}{\partial x_4}$, $D_{13}=x_4\frac{\partial}{\partial x_1}$, $D_{14}=x_4\frac{\partial}{\partial x_2}$, $D_{15}=x_4\frac{\partial}{\partial x_3}$. In the calculation of the vector fields of $\mathbb{P}^3$ it was taken into consideration that  the graded derivation $\sum_{i=1}^4x_i \frac{\partial}{\partial x_i}$ gives the zero vector field of $\mathbb{P}^3$.

In the affine chart $x_3=1$, these derivations are given by the following derivations of $k[x_1,x_2,x_4]$.  $D_1=x_1\frac{\partial}{\partial x_1}$, 
$D_2=x_2\frac{\partial}{\partial x_2}$, $D_3=x_4\frac{\partial}{\partial x_4}$, $D_4=x_1\frac{\partial}{\partial x_2}$, $D_5=x^2_1\frac{\partial}{\partial x_1}+ x_1x_2\frac{\partial}{\partial x_2}
+x_1x_4\frac{\partial}{\partial x_4}$, $D_6=x_1\frac{\partial}{\partial x_4}$, 
$D_7=x_2\frac{\partial}{\partial x_1}$, $D_8=x_1x_2\frac{\partial}{\partial x_1}+x^2_2\frac{\partial}{\partial x_2}+x_2x_4\frac{\partial}{\partial x_4}$, 
$D_9=x_2\frac{\partial}{\partial x_4}$, $D_{10}=\frac{\partial}{\partial x_1}$, $D_{11}=\frac{\partial}{\partial x_2}$, 
$D_{12}=\frac{\partial}{\partial x_4}$, $D_{13}=x_4\frac{\partial}{\partial x_1}$, $D_{14}=x_4\frac{\partial}{\partial x_2}$, $D_{15}=x_1x_4\frac{\partial}{\partial x_1}+x_2x_4\frac{\partial}{\partial x_2}+
x_4^2\frac{\partial}{\partial x_4}$. 

Let now $D=\sum_{i=1}^{15}\lambda_i D_i$ a derivation, $\lambda_i \in k$, $i=1,\ldots, 15$. The points $(0,0,0), (1,0,0),(0,1,0),(1,0,1),(0,1,1)$ are singular points of $X$ corresponding to the ideals $(x_1,x_2,x_4), (x_1+1,x_2,x_4), (x_1,x_2+1,x_4),(x_1+1,x_2,x_4+1),(x_1,x_2+1,x_4+1)$. A straightforward but a bit long calculation shows that the only derivation fixing these ideals is 
\begin{gather*}
D=\lambda(D_1+D_2+D_3+D_5+D_8)=\\
\lambda((x_1+x_1^2+x_1x_2)\frac{\partial}{\partial x_1}+(x_2+x_2^2+x_1x_2)\frac{\partial}{\partial x_2}+(x_4+x_1x_4+x_2x_4)\frac{\partial}{\partial x_4}).
\end{gather*}
However, this derivation does not fix the ideal $(x_1+1,x_2+1,x_4+a)$, where $a^2+a+1=0$,  corresponding to the singular point $(1,1,a)$, neither  the equation of $X$. Hence $X$ does not have any nontrivial global vector fields fixing all its singular points and therefore its minimal resolution has no non trivial vector fields and hence it has reduced automorphism scheme.
\end{example}

The main results of this section are the following two theorems. The first one gives an upper bound for the number of singularities of a canonically polarized surface $X$. The next one provides a condition under which a vector field fixes the singular points and lifts to the minimal resolution.

\begin{theorem}\label{sec11-th-1}
Let $X$ be a canonically polarized surface defined over a field of characteristic $p>0$. Suppose that $p$ does not divide $K_X^2$ and $X$ has a nontrivial global vector field. Let $f\colon X^{\prime}\rightarrow X$ be the minimal resolution of $X$. Let $\nu(P)$ be the number of $f$-exceptional curves over a point $P\in X$. Then
\begin{enumerate}
\item Suppose that $K_X^2=1$ and $p\not=2$. Then $\sum_{P\in X}\nu(P) \leq 55$.
\item Suppose that $K_X^2\geq 2$ and $p\not= 3$. Then
\[
\sum_{P\in X}\nu(P) \leq 12\chi(\mathcal{O}_X)+11K_X^2,
\]
\end{enumerate}
In particular, $X$ has at most $55$ singular points if $K_X^2=1$ and $12\chi(\mathcal{O}_X)+11K_X^2$ singular points if $K_X^2\geq 2$.
\end{theorem}

\begin{theorem}\label{sec11-th-2}
With assumptions as in Theorem~\ref{sec11-th-1}. Suppose also  that $p>5$ and 
\begin{enumerate}
\item $p>56$, if $K_X^2=1$,
\item $p>12\chi(\mathcal{O}_X)+11K_X^2 +1$,  if $K_X^2\geq2$.
\end{enumerate}
 Then
\begin{enumerate}
\item Every singular point of $X$ is a fixed point of $D$.
\item $D$ lifts to a vector field $D^{\prime}$ on the minimal resolution $X^{\prime}$ of $X$.
\item Every $f$-exceptional curve is an integral curve of $D^{\prime}$.
\end{enumerate}
\end{theorem}
\begin{remark}
The proof of the theorem uses Proposition~\ref{sec11-prop-5} which requires $p\not=3$ and the classification of rational double points in positive characteristic which requires $p>5$~\cite{Ar77}. In characteristic zero $\chi(\mathcal{O}_X)>0$ for any surface of general type and hence always 
$12\chi(\mathcal{O}_X)+11K_X^2 +1>24$. However in positive characteristic it is not known at the moment of this writing if $\chi(\mathcal{O}_X)>0$ for all $X$ and so it is possible that $12\chi(\mathcal{O}_X)+11K_X^2 +1$ may be 5 or less so $p=3,5$ must be excluded in the second case of the previous theorem, when $K_X^2 \geq 2$.
\end{remark}

Taking into consideration the classification of rational double points in positive characteristic~\cite{Ar77}, it immediately follows from Theorem~\ref{sec11-th-1} that

\begin{corollary}\label{sec11-cor-1}
With assumptions as in Theorem~\ref{sec11-th-1}. Suppose that the singular locus of $X$ consists of the points $A^{\ast}_i$ of type $A_{n_i}$, $i=1,\ldots, r$, $D^{\ast}_j$ of type $D_{m_j}$, $j=1,\ldots, s$, 
$E^{\ast}_{6,k}$ of type $E_6$, $k=1,\ldots , t$, $E^{\ast}_{7,\nu}$ of type $E_7$, $\nu=1,\ldots, w$ and $E^{\ast}_{8,\mu}$ of type $E_8$, $\mu=1,\ldots, u$. Then
\[
\sum_{i=1}^rn_i + \sum_{j=1}^s m_j +6t+7w+8u\leq 12\chi(\mathcal{O}_X)+11K_X^2.
\]
\end{corollary}

The proofs of Theorems~\ref{sec11-th-1},~\ref{sec11-th-2} will be given at the end of this section.

The next proposition is a simple generalization to the case of singular surfaces of a well known result on vector fields on smooth surfaces.
\begin{proposition}\label{sec11-prop-1}
Let $X$ be a Gorenstein normal projective surface and $D$ a nontrivial global vector field on $X$ such that $D^p=0$ or $D^p=D$. Let $\Delta$ be the divisorial part of $D$. Then there exists an exact sequence
\[
0 \rightarrow \mathcal{O}_X(\Delta)\rightarrow T_X \rightarrow \omega^{-1}(-\Delta) \rightarrow \mathcal{F} \rightarrow 0,
\]
where $\mathcal{F}$ is a zero dimensional coherent sheaf whose support is contained in the union of the singular points of $X$ and the isolated singularities of $D$.
\end{proposition}

\begin{proof}

Let $Z \subset X$ be the union of the singular points of $X$ and the isolated singularities of $D$. Then $Z$ is a finite set. Let $U=X-Z$. Then $U$ is smooth and the restriction of $D$ on $U$ has only divisorial singularities. Therefore the quotient of $U$ by $D$ is smooth~\cite{RS76}. Therefore there exists an exact sequence
\[
0\rightarrow \mathcal{O}_U(\Delta|_U)\rightarrow T_U \rightarrow L_U \rightarrow 0,
\]
where $L_U$ is an invertible sheaf on $U$~\cite[Proposition 1.9.3]{MP97}. Moreover, from the above sequence it follows that $L_U=\omega_U^{-1}(-\Delta|U)$. Applying $i_{\ast}$ in the above sequence, where $i \colon U \rightarrow X$ is the inclusion, and thaking into consideration that $\omega_X$ is invertible, we get an exact sequence
\[
0 \rightarrow \mathcal{O}_X(\Delta)\rightarrow T_X \rightarrow \omega^{-1}(-\Delta) \rightarrow \mathcal{F} \rightarrow 0,
\]
where $\mathcal{F}$ is a zero dimensional coherent sheaf whose support is contained in the union of the singular points of $X$ and the isolated singularities of $D$, as claimed.

\end{proof}

The next proposition gives a Riemann-Roch type inequality for divisors on surfaces with rational double points. 
\begin{proposition}\label{sec11-prop-2}
Let $X$ be a normal projective surface over an algebraically closed field $k$. Suppose that the singularities of $X$ are rational double points. Let $D$ be a divisor on $X$. Then
\[
\chi(\mathcal{O}_X(D))\leq \chi(\mathcal{O}_X)+\frac{1}{2}(D^2-K_X\cdot D).
\]
 \end{proposition}
\begin{remark}
The difference between the right hand side and the left hand side has been calculated explicitly with respect to the analytic type of the singularities of $X$ by M. Reid~\cite{Re85} in the case when the base field is $\mathbb{C}$. A similar calculation may be possible and desirable in positive characteristic.  However, for the purposes of this paper, the above inequality suffices. 
\end{remark}

\begin{proof}
Let $f \colon X^{\prime} \rightarrow X$ be the minimal resolution of $X$. Then the double dual $(f^{\ast}\mathcal{O}_X(D)))^{[1]}$ is invertible and hence $(f^{\ast}\mathcal{O}_X(D)))^{[1]}=\mathcal{O}_{X^{\prime}}(D^{\prime})$, where $D^{\prime}$ is a divisor on $X^{\prime}$. Now by~\cite{Ar85}, $f_{\ast}\mathcal{O}_{X^{\prime}}(D^{\prime})=\mathcal{O}_X(D)$ and $R^1f_{\ast} \mathcal{O}_{X^{\prime}}(D^{\prime})=0$. Therefore, 
\begin{gather}\label{sec11-eq-0}
\chi(\mathcal{O}_X(D))= \chi(\mathcal{O}_{X^{\prime}}(D^{\prime})).
\end{gather}
 Then by Rieman-Roch on $X^{\prime}$,
\begin{gather}\label{sec11-eq-1}
\chi(\mathcal{O}_{X^{\prime}}(D^{\prime}))=\chi(\mathcal{O}_{X^{\prime}})+\frac{1}{2}((D^{\prime})^2-K_{X^{\prime}}\cdot D^{\prime}).
\end{gather}
Since $X$ has rational double points and $X^{\prime}$ is the minimal resolution of $X$, $\chi(\mathcal{O}_{X^{\prime}})=\chi(\mathcal{O}_X)$ and $K_{X^{\prime}}=f^{\ast}K_X$. Moreover, it is clear that $f_{\ast}D^{\prime}=D$ and hence by the projection formula, 
\begin{gather}\label{sec11-eq-3}
K_{X^{\prime}}\cdot D^{\prime} =f^{\ast}K_X \cdot D^{\prime}=K_X \cdot C.
\end{gather}
Next I will relate $D^2$ and $(D^{\prime})^2$. Since $X$ has rational double points, $D$ is $\mathbb{Q}$-Cartier. Let $m\in \mathbb{Z}$ be a positive integer such that $mD$ is  Cartier. Then 
\[
mD^{\prime}=f^{\ast}(mD)+ F,
\]
where $F$ is a divisor supported on the exceptional set of $f$. Then,
\[
m^2(D^{\prime})^2=m^2D^2+F^2<m^2D^2,
\]
since $F^2<0$. Hence $(D^{\prime})^2<D^2$. Now the statement of the proposition follows from this and the equations (\ref{sec11-eq-0}), (\ref{sec11-eq-1}) and (\ref{sec11-eq-3}).
\end{proof}

The next result relates the first cohomology of the tangent sheaf of a rational double point with the number of exceptional divisors over it in the minimal resolution.

\begin{proposition}\label{sec11-prop-3}
Let $P\in X$ be a rational double point singularity. Let $f \colon X^{\prime} \rightarrow X$ be its minimal resolution and $E_i$, $i=1,\ldots, n$ the $f$-exceptional curves . Then
\[
h^1(T_{X^{\prime}}) \geq n.
\]
\end{proposition}
\begin{remark}
If the characteristic of the base field is zero then the inequality in the previous proposition is in fact equality~\cite{BW74}.
\end{remark}

\begin{proof}
The proof follows the lines of the proof of~\cite[Pages 70, 71]{BW74} with some modifications to deal with the possible positive characteristic complications.

Let $Z=\sum_{i=1}^n m_i E_i$ an integral  effective divisor supported on the exceptional set of $f$. Then for sufficiently large $m_i$, $i=1,\ldots, n$, $-Z$ is $f$-ample. Therefore, $H^i(T_{X^{\prime}}(-mZ))=0$, for $m>>0$, $i=1,2$. Taking now cohomology on the exact sequence
\[
0 \rightarrow T_{X^{\prime}}(-mZ) \rightarrow T_{X^{\prime}} \rightarrow T_{X^{\prime}}\otimes \mathcal{O}_{mZ} \rightarrow 0,
\]
it follows that
\begin{gather}\label{sec11-eq-4}
H^1(T_{X^{\prime}})=H^1(T_{X^{\prime}}\otimes \mathcal{O}_{mZ}).
\end{gather}
Let $E=\sum_{i=1}^n E_i$ be the reduced $f$-exceptional divisor. Then there exists an exact sequence
\[
0 \rightarrow N \rightarrow \mathcal{O}_{mZ} \rightarrow \mathcal{O}_E \rightarrow 0,
\]
where $N$ is supported on the exceptional set of $f$. Then the previous sequence gives the exact sequence
\[
0 \rightarrow T_{X^{\prime}}\otimes N \rightarrow T_{X^{\prime}} \otimes \mathcal{O}_{mZ} \rightarrow T_{X^{\prime}} \otimes \mathcal{O}_E \rightarrow 0.
\]
After taking cohomology in the previous sequence, and since $N$ has 1-dimensional support, it follows that
\begin{gather}\label{sec11-eq-5}
h^1(T_{X^{\prime}}) \geq h^1( T_{X^{\prime}} \otimes \mathcal{O}_E).
\end{gather}
Next, there exists an exact sequence
\begin{gather}\label{sec11-eq-6}
0 \rightarrow T_E \rightarrow T_{X^{\prime}} \otimes \mathcal{O}_E \stackrel{\sigma}{\rightarrow} \oplus_{i=1}^n\mathcal{N}_{E_i} \rightarrow 0,
\end{gather}
where the map $\sigma$ is the sum of the composition of the natural maps $T_{X^{\prime}}\otimes \mathcal{O}_E \rightarrow T_{X^{\prime}}\otimes \mathcal{O}_{E_i} $ and $ T_{X^{\prime}}\otimes  \mathcal{O}_{E_i} \rightarrow \mathcal{N}_{E_i}$, $i=1,\ldots, n$. The exactness of the sequence above can easily be checked locally. 

Now since $P \in X$ is a rational double point, $E_i \cong \mathbb{P}^1$ and $\mathcal{N}_{E_i}\cong \mathcal{O}_{\mathbb{P}^1}(-2)$, $i=1,\ldots, n$. The proposition now follows from the equation (\ref{sec11-eq-5}) and by taking cohomology in (\ref{sec11-eq-6}).

Finally I would like to mention that in~\cite{BW74}, the equality in the statement of the  proposition is proved by taking the exact sequence (\ref{sec11-eq-6}) with $mZ$ in the place of $E$ on the left hand side of the sequence and then using a result by Tjurina that $H^1(T_{mZ})=0$. However, this is proved only in characteristic zero and moreover, the exact (\ref{sec11-eq-6}) may not be exact with $mZ$ in the place of $E$ if some of the coefficients of $mZ$ are divisible by $p$.  

\end{proof}

The next proposition gives a bound for the number of singular points of a projective surface  with rational double points and a nontrivial  global vector field.

\begin{proposition}\label{sec11-prop-4}
Let $X$ be a normal projective surface over an algebraically closed field of characteristic $p>0$ with rational double point singularities. Suppose that $X$ has a nontrivial global vector field $D$ such that $D^p=0$ or $D^p=D$. Let $f \colon X^{\prime} \rightarrow X$ be the minimal resolution of $X$. Then
\[
\sum_{P\in X}\nu(P) \leq 12 \chi(\mathcal{O}_X)-K_X^2+\Delta^2+K_X \cdot \Delta,
\]
where $\Delta$ is the divisorial part of $D$ and $\nu(P)$ is the number of $f$-exceptional curves over $P\in X$. 
\end{proposition} 
\begin{remark}
If $\Delta =0$, a case that frequently happens, then $12\chi(\mathcal{O}_X)-K_X^2$ is a bound for the singular points of $X$, a bound which is a function of only numerical invariants of $X$. A similar bound will be given later without the assumption $\Delta=0$ if $K_X$ is ample and $p$ does not divide $K_X^2$.
\end{remark}

\begin{proof}
There exists a natural exact sequence
\[
0 \rightarrow f_{\ast}T_{X^{\prime}} \rightarrow T_X \rightarrow N \rightarrow 0,
\]
where $N$ is a zero dimensional coherent sheaf on $X$ supported on the singular locus of $X$. Hence $\chi(N)=h^0(N)\geq 0$. Then from the above sequence it follows that
\begin{gather}\label{sec11-eq-7}
\chi(f_{\ast}T_{X^{\prime}})\leq \chi (f_{\ast}T_{X^{\prime}})+\chi(N)=\chi(T_X).
\end{gather}
From the Leray spectral sequence and considering that $f$ is birational with at most one dimensional fibers we get the exact sequence
\[
0 \rightarrow H^1(f_{\ast}T_{X^{\prime}})\rightarrow H^1(T_{X^{\prime}}) \rightarrow H^0(R^1f_{\ast}T_{X^{\prime}})\rightarrow H^2(f_{\ast}T_{X^{\prime}})\rightarrow H^2(T_{X^{\prime}}) \rightarrow 0.
\]
Counting dimensions we get that
\begin{gather}\label{sec11-eq-8}
\chi(f_{\ast}T_{X^{\prime}})=\chi(T_{X^{\prime}})+h^0(R^1f_{\ast}T_{X^{\prime}})
\end{gather}
Now from Propositions~\ref{sec11-prop-1},~\ref{sec11-prop-2} it follows that
\begin{gather}\label{sec11-eq-9}
\chi (T_X)=\chi(\mathcal{O}(\Delta))+\chi(\omega^{-1}_X(-\Delta))-\chi(\mathcal{F})\leq \chi(\mathcal{O}(\Delta))+\chi(\omega^{-1}_X(-\Delta)) \leq \\\nonumber
2\chi(\mathcal{O}_X)+\frac{1}{2}(\Delta^2-K_X\cdot \Delta)+\frac{1}{2}((K_X+\Delta)^2+K_X\cdot (K_X+\Delta))=\\\nonumber
2\chi(\mathcal{O}_X)+K_X^2+K_X\cdot \Delta +\Delta^2.
\end{gather}
Now from the equations (\ref{sec11-eq-7}), (\ref{sec11-eq-8}) and (\ref{sec11-eq-9}) we get that
\begin{gather}\label{sec11-eq-10}
\chi(T_{X^{\prime}})+h^0(R^1f_{\ast}T_{X^{\prime}})\leq 2\chi(\mathcal{O}_X)+K_X^2+K_X\cdot \Delta +\Delta^2.
\end{gather}

Then by Proposition~\ref{sec11-prop-3} and the previous inequality we get that
\begin{gather}\label{sec11-eq-12}
\chi(T_{X^{\prime}})+\sum_{P\in X}\nu(P) \leq \chi(T_{X^{\prime}})+h^0(R^1f_{\ast}T_{X^{\prime}})\leq 2\chi(\mathcal{O}_X)+K_X^2+K_X\cdot \Delta +\Delta^2.
\end{gather}
Now by the Riemann-Roch on $X^{\prime}$, Noether's formula and the facts that $K_{X^{\prime}}=f^{\ast}K_X$, $\chi(\mathcal{O}_{X^{\prime}})=\chi(\mathcal{O}_X)$ (since $X$ has rational double point singularities), we get that
\begin{gather}\label{sec11-eq-11}
\chi(T_{X^{\prime}})=\frac{7}{6}K_{X^{\prime}}^2-\frac{5}{6}c_2(X^{\prime})=\frac{7}{6}K_{X^{\prime}}^2-\frac{5}{6}(12\chi(\mathcal{O}_{X^{\prime}})-K_{X^{\prime}}^2)=
-10\chi(\mathcal{O}_{X^{\prime}})+2K_X^2.
\end{gather}
Now the statement of the proposition follows immediately from the equations (\ref{sec11-eq-12}) and (\ref{sec11-eq-11}).
\end{proof}

The following  lemma is an easy generalization of the Hodge index theorem to surfaces with rational double points. It will be used throughout this paper.

\begin{lemma}\label{hodge} 
Let $X$ be a normal projective surface with rational double points. Let $A$ be a nef and big line bundle on $X$ and $C$ a divisor on $X$. Then 
\[
C^2 A^2 \leq (C\cdot A)^2.
\]
\end{lemma}
\begin{proof}
Let $X^{\prime}\rightarrow X$ be the minimal resolution of $X$. Since $X$ has rational double point singularities, $C$ is $\mathbb{Q}$-Cartier. Let $m>0$ be an integer such that $mC$ is Cartier. Then, since $f^{\ast}A$ is also nef and big on $X^{\prime}$ and the generalized Hodge index theorem for nef and big line bundles~\cite[Corollary 2.4]{Ba01}, it follows that 
\[
m^2C^2 \cdot A^2=(f^{\ast}(mC))^2\cdot (f^{\ast}A)^2 \leq ( f^{\ast}(mC) \cdot f^{\ast}A)^2=m^2(C\cdot A)^2.
\]
From this the lemma follows immediately.
\end{proof}

The following proposition is the last ingredient needed in order to prove Theorems~\ref{sec11-th-1},~\ref{sec11-th-2}. It will also be needed later for the proof of the main theorem of this paper.

\begin{proposition}\label{sec11-prop-5}
Let $X$ be a canonically polarized surface defined over a field of characteristic $p>0$. Suppose that $X$ has a nontrivial global vector field $D$  such that $D^p=0$ or $D^p=D$ and such that  $p$ does not divide $K_X^2$. Then
\begin{enumerate}
\item  Suppose that $K_X^2=1$ and $p\not=2$. Then $K_X\cdot \Delta \leq 4$ and $\Delta^2\leq 16$.
\item Suppose that $K_X^2\geq 2$ and $p\not= 3$. Then
\begin{gather}
K_X\cdot \Delta \leq 3K_X^2,\\
\Delta^2 \leq 9K_X^2,\nonumber
\end{gather}
\end{enumerate}
where $\Delta $ is the divisorial part of $D$.
\end{proposition}

\begin{proof}

Let $\pi \colon X \rightarrow Y$ be the quotient of $X$ by the $\alpha_p$ or $\mu_p$ action on $X$ defined by $D$. Then $\pi$ is a purely inseparable map of degree $p$ and by~\cite{RS76}, $K_X=\pi^{\ast}K_Y+(p-1)\Delta$ (this formula holds by~\cite{RS76} in the smooth part of $Y$ and hence everywhere since $Y$ is normal).

By~\cite[Theorem 1.20]{Ek88}, the linear system $|nK_X|$ is base point free for $n=3$ if $K_X^2=1$ and $n=2$ if $K_X^2\geq 2$. 

Suppose that $K_X^2\geq 2$ and hence $n=2$. The proof in the case when $K_X^2=1$ and $n=3$ is identical and is omitted. Then by~\cite[Theorem 6.3]{Jou83},~\cite{Za44}, the general member of $|3K_X|$ is of the form $p^{\nu} C$, where $C$ is an irreducible and reduced curve. Since $p$ does not divide $K_X^2$, $\nu=0$ and hence the general member of $|3K_X|$ is a reduced and irreducible curve. 

Therefore  there exists $C\in |3K_X|$ such that $C$ is reduced and irreducible and it does not pass through any singular point of $X$ or isolated singularity of $D$. Let $\tilde{C}=\pi(C)$. Then, since $C$ is in the smooth part of $X$ and does not contain any isolated singularity of $D$, $\tilde{C}$ lies in the smooth part of $Y$.

Suppose that $C$ is an integral curve of $D$. Then~\cite{RS76}, $\pi^{\ast}\tilde{C}=C$ and therefore $C^2=p\tilde{C}^2=pm$, $m\in \mathbb{Z}$ since $\tilde{C}$ is in the smooth part of $Y$. Then, 
since $C\in |3K_X|$, it follows that $p$ divides $9K_X^2$ and hence, since $p\not= 3$, $p$ divides $K_X^2$, which is impossible.

Hence $C$ is not an integral curve of $D$ and hence the map $\pi \colon C \rightarrow \tilde{C} $ is birational. Moreover~\cite{RS76}, $\pi^{\ast}\tilde{C}=pC$. Now since $\tilde{C}$ is contained in the smooth part of $Y$, adjunction for $\tilde{C}$ holds and hence
\begin{gather*}
2p_a(\tilde{C})-2=K_Y\cdot \tilde{C} +\tilde{C}^2=\pi^{\ast}K_Y\cdot C+pC^2=K_X\cdot C-(p-1)\Delta \cdot C +pC^2=\\
(K_X\cdot C+C^2)+(p-1)(C^2-\Delta \cdot C)=2p_a(C)-2+(p-1)(9K_X^2-3K_X\cdot \Delta).
\end{gather*}

Since the map $C\rightarrow \tilde{C}$ is birational, it follows that $p_a(\tilde{C}) \geq p_a(C)$. Then the above equation gives that $3K_X^2-K_X\dot \Delta \geq 0$ and hence $K_X \cdot \Delta \leq 3K_X^2$, as claimed.

Finally, since $K_X$ is ample, it follows from Lemma~\ref{hodge} that 
\[
\Delta^2 \leq \frac{(\Delta\cdot K_X)^2}{K_X^2}\leq \frac{(3K_X^2)^2}{K_X^2}= 9K_X^2,
\]
as claimed.

\end{proof}

We are now in a position to prove Theorems~\ref{sec11-th-1},~\ref{sec11-th-2}.

\begin{proof}[\textbf{Proof of Theorem~\ref{sec11-th-1}}]
Since $X$ has a nontrivial global vector field, it follows from~\cite{RS76} that $X$ has a nontrivial global vector field $D$ such that $D^p=0$ or $D^p=D$. Then the statement of the theorem follows immediately from Propositions~\ref{sec11-prop-4},~\ref{sec11-prop-5}. In the case when $K_X^2=1$, one must also use that fact that $1\leq \chi(\mathcal{O}_X) \leq 3$~\cite{Li09}.
\end{proof}

\begin{proof}[\textbf{Proof of Theorem~\ref{sec11-th-2}}]
I will only do the case when $K_X^2\geq 2$. The case when $K_X^2=1$ is identical and is omitted.

By assumption $X$ has canonical singularities and hence its singularities are rational double points. Since $p>5$, the equations classifying the rational double points are the same as those in characteristic zero~\cite{Ar77}. Hence $X$ may have either singularities of type $A_n$, $D_m$, $E_6$, $E_7$ and $E_8$. Then, by Theorem~\ref{sec11-th-1}, if $12\chi(\mathcal{O}_X)+11K_X^2+1<p$, $n+1<p$ and $m<p$. The statement of the theorem is local at the singularities. In order to prove the theorem consider cases with respect to the singularities of $X$.

Let $P\in X$ be a singular point of $X$.  I will do in detail only the case when $P\in X$ is of type $A_n$. The rest are similar and are left to the reader.

By passing to the completion at $P$, we may assume that $X$ is given by $xy+z^{n+1}=0$. Moreover, by the assumptions and Theorem~\ref{sec11-th-1},  $n+1 <p$. $D$ is induced by a derivation $D$ of 
$k[[x,y,z]]$ such that $D(xy+z^{n+1})\in (xy+z^{n+1})$. Now
\[
D(xy+z^{n+1})=xDy+yDx+(n+1)z^nDz,
\]
with $n+1\not=0$. From the above equation it follows that $yDx\in (x,z)$ and hence, since $y\not\in (x,z)$, it follows that $Dx\in (x,z)\subset (x,y,z)$. Similarly, $Dy\in (y,z)$. Finally, from the previous equation it follows that $z^nDz\in (x,y,z^{n+1})$. If $Dz\not\in (x,y,z)$, then $Dz$ is a unit in $k[[x,y,z]]$ and hence $z^n\in (x,y,z^{n+1})$, which is impossible. Hence in this case, $P$ is a fixed point of $D$.

Next I will show that $D$ lifts to the minimal resolution $f \colon X^{\prime}\rightarrow X$ of $X$. Since $X$ has rational double points, $f$ is obtained by successively blowing up the singular points. Let $f_1 \colon X_1 \rightarrow X$ be the blow up of all singular points of $X$. Then, since the singular points of $X$ are fixed points of $D$, $D$ lifts to a vector field $D_1$ on $X_1$. Moreover, $X_1$ has also rational double points, of simpler type that those of $X$. Then, the previous argument shows that the singular points of $X_1$ are fixed points of $D_1$. Then one can blow up again and continue this process until the minimal resolution is reached and therefore $D$ lifts to a vector field $D^{\prime}$ on $X^{\prime}$.

 It remains to show that every $f$-exceptional curve is an integral curve of $D^{\prime}$. In order to prove this it suffices to prove, since $f$ is a composition of blow ups, the following. Let $P\in Z$ be a rational double point on a surface $Z$  which is a fixed point of a vector field $D$ of $Z$ and let $g \colon \tilde{Z} \rightarrow Z$ be the blow up of $P$. Then the reduced $g$-exceptional curves are integral curves of $\tilde{D}$, the vector field on $Z$ lifting $D$.  
 
 Let $E$ be a $g$-exceptional curve.  Suppose that $P\in X$ is of type $A_n$. Then $f^{-1}(P)=E_1+E_2$, where where $E_1, E_2$ are distinct smooth rational curves. Suppose that $P \in X$ is of one of the types  $D_n$, $E_6$, $E_7$, $E_8$. Then $f^{-1}(P)=2E$, where $E$ is a smooth rational curve. Then the claim that the $g$-exceptional curves are integral curves of $\tilde{D}$ is an immediate consequence of Lemma~\ref{sec11-lemma-1} which follows.
 
 This concludes the proof of Theorem~\ref{sec11-th-2}.
 
\begin{lemma}\label{sec11-lemma-1}
Let $f \colon X \rightarrow Y$ be a morphism between varieties defined over an algebraically closed field $k$ of characteristic $p>0$ Such that $X$ is normal. Suppose that $D_Y$ is a nontrival global vector field on $Y$ and $D_X$ a nontrivial global vector field on $X$ lifting $D_Y$, i.e., there exists a commutative diagram 
\[
\xymatrix{
f_{\ast}\mathcal{O}_X \ar[r]^{D_X} & f_{\ast}\mathcal{O}_X\\
\mathcal{O}_Y \ar[u]\ar[r]^{D_Y} & \mathcal{O}_Y \ar[u]
}
\]
Let $P\in Y$ be a fixed point of $D_Y$ and $[f^{-1}(P)]=\sum_{i=1}^n m_i Z_i$, be the cycle corresponding to the fiber $f^{-1}(P)$. Let $Z_i$ be a codimension 1 component such that $p$ does not divide $m_i$, Then $D_X(I_{Z_i}) \subset I_{Z_i} $, i.e., $Z_i$ is stabilized by $D_X$.
\end{lemma}

\begin{proof}
Let $Z_i$ be a codimension 1 component of $f^{-1}(P)$ such that $p$ does not divide $m_i$. In order to prove that $D_X(I_{Z_i})\subset I_{Z_i}$ it suffices to prove this in an affine open set $U$ of $X$ which is contained in the smooth locus of $X$ and such that $U \cap Z_i \not= \emptyset$. Therefore the proof is reduced to the case when both $X$ and $Y$ are affine.  Let then $Y=\mathrm{Spec} A$, $X =\mathrm{Spec}B$. Then $D_Y$, $D_X$ are induced by derivations of $A$ and $B$, respectively.  Let $\mathbf{m}_p\subset A$ be the maximal ideal corresponding to $P$. Then $f^{-1}(P)$ is given by the ideal $\mathbf{m}_PB$ of $B$. Moreover, since $D_X$ lifts $D_Y$ and $D_Y(\mathbf{m}_P) \subset \mathbf{m}_P$, it follows easily that  $D_X(\mathbf{m}_PB) \subset \mathbf{m}_PB$. Then if $U$ is chosen small enough, 
$\mathbf{m}_PB=I_{Z_i}^{m_i}$. Moreover, since $X$ is normal and $U$ is in the smooth part of $X$, $I_{Z_i}$ is a prime ideal of $B$ and $I_{Z_i}=(b)$, for some $b\in B$. Then 
\[
D_X(b^{m_i}) =m_ib^{m_{i-1}}D_Xb\in (b^{m_i}).
\]
Since $p$ does not divide $m_i$ it follows that $b^{m_{i-1}}D_Xb\in (b^{m_i})$ and hence $b^{m_{i-1}}D_Xb =b^{m_i}c$ and therefore $D_Xb \in (b)$. Hence $D_X(I_{Z_i}) \subset I_{Z_i}$, as claimed.

\end{proof}

\end{proof}


\section{Integral curves and fixed points of vector fields on surfaces.}\label{sec-2}
Let $X$ be a normal projective surface defined over an algebraically closed field $k$ of characteristic $p>0$. Let $D$ be a nontrivial vector field on $X$ (or equivalently a $k$-derivation of $\mathcal{O}_X$). This section contains various properties of integral curves of $D$ which are needed for the proofs of the main results of this paper.



The next proposition presents a method  to find integral curves of $D$.

\begin{proposition}[Proposition 2.1~\cite{Tz18}]\label{sec1-prop1}
Suppose that either $D^p=0$ or $D^p=D$. Then $D$ induces an  $\alpha_p$ or $\mu_p$ action on $X$, respectively. Let $\pi \colon X \rightarrow Y$ be the quotient of $X$ by this action. Let $L$ be a rank one reflexive sheaf on $Y$ and $M=(\pi^{\ast}L)^{[1]}$. Then $D$ induces a $k$-linear map
\[
D^{\ast} \colon H^0(X,M) \rightarrow H^0(X,M)
\]
with the following properties:
\begin{enumerate}
\item $\mathrm{Ker}(D^{\ast})=H^0(Y,L)$ (considering $H^0(Y,L)$ as a subspace of $H^0(X,M)$ via the map $\pi^{\ast}$).
\item If $D^p=0$ then $D^{\ast}$ is nilpotent and if $D^p=D$ then $D^{\ast}$ is a diagonalizable map whose eigenvalues are in the set $\{0,1,\ldots,p-1\}$.
\item Let $s\in H^0(X,M)$ be an eigenvector of $D^{\ast}$. Then $D(I_{Z(s)})\subset I_{Z(s)}$, where  $Z(s)$ is the divisor of zeros of $s$. In particular, if $D^{\ast}(s)=\lambda s$, and $\lambda\not= 0$, then $(D(I_{Z(s)}))|_V=I_{Z(s)}|_V$, where $V=X-\pi^{-1}(W)$, $W\subset Y$ is the set of points that $L$ is not free.
 \end{enumerate}
\end{proposition}

The previous proposition shows that every eigenvector of $D^{\ast}$ corresponds to a curve $C\subset X$ such that $D(I_C)\subset I_C$ and therefore $D$ induces a vector field on $C$. However it is possible that $D(\mathcal{O}_X)\subset I_C$ and hence the induced vector field on $C$ is trivial. This implies that $C$ is contained in the divisorial part of $D$. This cannot happen of course if $D$ has only isolated singularities. 

Let $C=n_1C_1+\cdots +n_kC_k$ be a curve in $X$ and its decomposition into its prime components. Suppose that $D(I_C)\subset I_C$. In general $D$ does not induce vector fields on $C_i$, i.e, $D(I_{C_i})$ may not be contained in $I_{C_i}$. For example for any reduced and irreducible curve  $C$, $D$ stablizes $pC$ but not necessarily $C$. The next proposition provides some conditions in order for $D$ to restrict to $C_i$.

\begin{proposition}\label{sec1-prop2}
Let $C \subset X$ be a curve such that $D(I_C)\subset I_C$, where $I_C \subset \mathcal{O}_X$ is the ideal sheaf of $C$ in $X$. Let $C=n_1C_1+\cdots +n_kC_k$ be the decomposition of $C$ in its irreducible and reduced components. If $p$ does not divide $n_i$, for all $1\leq i \leq k$, then $D(I_{C_i})\subset I_{C_i}$, for all $1\leq i \leq k$. Therefore $D$ stabilizes the reduced part of every irreducible component of $C$ and hence induces a vector field on $C_i$, for all $1\leq i \leq k$.
\end{proposition}

\begin{proof}
Let $i\in \{1,\ldots, k\}$. In order to prove that $D(I_{C_i})\subset I_{C_i}$ it suffices to show this on a nonempty open subset $U$ of $X$ such that $U \cap C_i \not= \emptyset$. In fact, by taking $U$ small enough we may assume that $U \cap C_j =\emptyset$, for all $j\not= i$. Hence we may assume that $X=\mathrm{Spec}A$ is affine and smooth and $C=n_iC_i$. Hence $I_C=(t^{n_i})$, for some $t\in A$ and $I_{C_i}=(t)$. $D$ is induced by a $k$-derivation of $A$. Then since $D(I_C)\subset I_C$, it follows that $n_it^{n_i-1}Dt \in (t^{n_i})$ and hence there exists $a\in A$ such that $n_it^{n_i-1}Dt =at^{n_i}$. Now since $p$ does not divide $n_i$, $n_i \not=0$ in $k$ and hence it follows that $Dt \in (t)$. Hence $D(I_{C_i})=I_{C_i}$, as claimed.

\end{proof}

\begin{corollary}\label{sec1-cor1}
With assumptions as in Proposition~\ref{sec1-prop2}. Suppose  in addition that $K_X$ is an ample invertible sheaf and  $K_X \cdot C < p$. Then $D(I_{C_i})\subset I_{C_i}$, for all $1\leq i \leq k$. Therefore $D$ stabilizes the reduced part of every irreducible component of $C$ and hence induces a vector field on $C_i$, for all $1\leq i \leq k$.
\end{corollary}

\begin{proof}
Since $K_X$ is assumed to be ample and invertible, the condition $K_X\cdot C <p$ immediately implies that $n_i <p$, for all $1\leq i \leq k$. Then the corollary follows directly from Proposition~\ref{sec1-prop2}.
\end{proof}

\begin{proposition}\label{sec1-prop3}
Suppose that $X$ is $\mathbb{Q}$-factorial and $K_X$ is an ample invertible sheaf. Let $C\in |mK_X|$ be a curve such that $D(I_C) \subset I_C$. Let $C=n_1C_1+\cdots +n_kC_k$ its decomposition into its reduced and irreducible components. Suppose that $K_X^2< p/(m^2+3m)$. Let $P\in C_i \cap C_j$, $i\not= j$, be a closed point such that $P\in X$ is smooth. Then $P$ is a fixed point of $D$. 

\end{proposition}

\begin{proof}
By Corollary~\ref{sec1-cor1}, $D(I_{C_i})\subset I_{C_i}$, for all $1\leq i \leq k$. The result is local at $P$. Let $U=\mathrm{Spec} A$ be an affine open subset of $X$ containing $P$ but no other point of $C_i \cap C_j$. Since $P\in X$ is a smooth point, $U$ may be taken to be smooth. Let $I$ and $J$ be the ideals of $C_i$ and $C_j$ respectively. Then $(I+J)|_U=Q$, with $r(Q)=\mathbf{m}_P$, the maximal ideal corresponding to the point of intersection $P$ of $C_i$ and $C_J$. Now since $D(I) \subset I$ and $D(J) \subset J$, it follows that $D(I+J)=D(I)+D(J)\subset I+J$. Hence $D(Q)\subset Q$. I will show that this implies that $D(\mathbf{m}_P)\subset \mathbf{m}_P$ and therefore $P$ is a fixed point of $D$.

In order to show that $D(\mathbf{m}_P)\subset \mathbf{m}_P$, I will first show that $C_i\cdot C_j <p$. Then if $I=(f)$ and $J=(g)$, $f,g \in A$, $\dim_k A/(f,g)< p$. Hence for any $a\in \mathbf{m}_P$, there exists $\nu<p$ such that $a^{\nu} \in Q=I+J$. Let ${\nu}_0<p$ be the smallest such $\nu$. Then $D(a^{\nu_0})=\nu_0a^{\nu_0-1}Da \in Q=I+J$. $Q$ is a primary ideal and $a^{\nu_0-1} \not\in Q$. Hence $(Da)^s \in Q \subset \mathbf{m}_P$, for some $s \geq 0$. Hence $Da \in \mathbf{m}_P$. Therefore $D(\mathbf{m}_P)\subset \mathbf{m}_P$, as claimed. 

It remains to show that $C_i \cdot C_j <p$. By definition, $mK_X \sim \sum_{s=1}^k n_s C_s$.  Let $1\leq i,j \leq k$. Then
\begin{gather}\label{sec1-prop3-eq1}
mK_X\cdot C_i=n_j C_i \cdot C_j +n_i C_i^2 +\sum_{s\not= i,j} n_s C_s \cdot C_i \geq n_j C_i \cdot C_j +n_iC_i^2.
\end{gather}
On the other hand, $mK_X^2=\sum_{s=1}^m n_s K_X\cdot C_s$ and since $K_X$ is ample, it follows that $K_X\cdot C_s>0$ for every $1\leq s \leq m$ and therefore $K_X\cdot C_s \leq n_s K_X \cdot C_s \leq mK_X^2$. Then from (\ref{sec1-prop3-eq1}) it follows that
\begin{gather}\label{sec1-prop3-eq2}
C_i\cdot C_j \leq m^2K_X^2 -n_iC_i^2.
\end{gather}
Next I will show that $-C_i^2\leq 2+K_X \cdot C_i$. Let $f \colon X^{\prime}\rightarrow X$ be the minimal resolution of $X$. Let $C_i^{\prime}=f_{\ast}^{-1}C_i$, be the birational transform of $C_i$ in $X^{\prime}$. Then by the adjunction formula for $C_i^{\prime}$ it follows that 
\begin{gather}\label{sec1-eq-1000}
-(C_i^{\prime})^2=-2p_a(C_i^{\prime})+2+K_{X^{\prime}}\cdot C_i^{\prime}\leq 2+K_{X^{\prime}}\cdot C_i^{\prime}.
\end{gather} 
Now there are adjunction formulas
\begin{gather}
f^{\ast}C_i=C^{\prime}_i+ E\\
K_{X^{\prime}}+F=f^{\ast}K_X \nonumber
\end{gather}
Where $E$ and $F$ are effective $f$-exceptional divisors ($F$ is effective because $f$ is the minimal resolution). From these immediately follows that $C_i^2\geq (C_i^{\prime})^2$ and $K_X\cdot C_i \geq K_{X^{\prime}}\cdot C_i^{\prime}$. From these and the equation (\ref{sec1-eq-1000}) it follows that
\begin{gather}\label{sec1-eq-8}
-C_i^2\leq 2+K_X\cdot C_i.
\end{gather}
Then from the equation (\ref{sec1-prop3-eq2}) it follows that 
\begin{gather}\label{sec1-prop3-eq3}
C_i\cdot C_j \leq m^2K_X^2+2n_i+n_iK_X\cdot C_i.
\end{gather}
But it has been shown earlier that $n_iK_X\cdot C_i \leq mK_X^2$ and hence $n_i \leq mK_X^2$ and $K_X\cdot C_i <mK_X^2$. Hence 
\begin{gather}\label{sec2-eq-1114}
C_i \cdot C_j \leq (m^2+3m)K_X^2<\frac{p(m^2+3m)}{m^2+3m}<p,
\end{gather}
as claimed. This concludes the proof.

\end{proof}
The proof of the previous proposition shows also the following.
\begin{corollary}\label{sec1-cor0}
Let $C_1$, $C_2$ be two different irreducible and reduced curves on $X$ such that $D(I_{C_i})\subset I_{C_i}$, for $i=1,2$. Assume that $C_1 \cdot  C_2 <p$. Then every point of intersection of $C_1$ and $C_2$ which is a smooth point of $X$ is a fixed point of $D$.
\end{corollary}
\begin{remark}
Proposition~\ref{sec1-prop3} and Corollary~\ref{sec1-cor0} apply in particular in the case when the singularities of $X$ are rational double points since they are $\mathbb{Q}$-factorial.
\end{remark}

As explained in Section~\ref{sec-1}, in general, in positive characteristic a vector field on a variety $Y$ does not fix its singular points. In section~\ref{sec-1} conditions were obtained which imply that a vector field on a surface fixes its singular points. The next proposition shows gives a condition which implies that a vector field on a curve fixes the singular points of the curve.

\begin{proposition}\label{sec1-prop4}
Let $D$ be a nontrivial vector field of either additive or multiplicative type on a smooth surface $X$ defined over an algebraically closed field $k$ of characteristic $p>0$. Let $C \subset X$ be a reduced and irreducible curve such that $D(I_C)\subset I_C$, where $I_C$ is the ideal sheaf of $C$ in $X$. Suppose that $\mathrm{p}_a(C)<(p-1)/2$. Then $D$ fixes every singular point of $C$ and lifts to the normalization 
$\bar{C}$ of $C$.
\end{proposition}

\begin{proof}
We may assume that $D(\mathcal{O}_X)\not\subset I_C$ and hence the restriction of $D$ on $C$ is not trivial (otherwise the result is obvious).

Let $\pi \colon X \rightarrow Y$ be the quotient of $X$ by the $\alpha_p$ or $\mu_p$ action on $X$ induced by $D$. Then $\pi$ is a purely inseparable morphism of degree $p$. Let $\tilde{C}=\pi(C)\subset Y$. Then $C=\pi^{\ast}\tilde{C}$ and $\pi_{\ast}C=p\tilde{C}$~\cite{RS76}. Let $P\in C$ be a singular point of $C$ and $Q=\pi(P)\in Y$. If $P$ is a fixed point of $D$ then there is nothing to prove. Suppose that $P$ is not a fixed point of $D$. Then $Q\in Y$ is a smooth point of $Y$~\cite{AA86}. Hence locally around $Q \in Y$, $X\rightarrow Y$ is an $\alpha_p$ or $\mu_p$ torsor and hence the same holds for $C\rightarrow \tilde{C}$. Consider cases with respect to whether $Q \in \tilde{C}$ is a singular or a smooth point of $C$.

\textbf{Case 1.} $Q\in \tilde{C}$ is singular. Then since $P\in X$ is not a fixed point of $D$, in suitable local analytic coordinates at $P$, $\mathcal{O}_X=k[[x,y]]$, $D=h(x,y) \partial/\partial x$ and $\mathcal{O}_Y=k[[x^p,y]]$~\cite[Theorem 1]{RS76}. Then $I_{\tilde{C}}=(f(x^p,y))$ and since it is assumed that $Q\in \tilde{C}$ is singular, $f(x^p,y)\in (x^p,y)^2$. Then $I_C=(f(x^p,y)) \subset k[[x,y]]$. Write $f(x^p,y)=\sum_i f_i(x^p)y^i$. Then either $m_P(f(x^p,y)) \geq p$ (considered in $k[[x,y]]$) or there exists an $m\geq 1$ such that $f_m(x^p)$ is a unit in $k[[x^p]]$. 

The first case is easily seen to be impossible since $C$ is assumed to have arithmetic genus less than $p$ and a curve of arithmetic genus less than $p$ cannot have a point of multiplicity bigger than $p$.

Suppose then that there exists an $m \geq 1$ such that $f_m(x^p)$ is a unit in $k[[x^p]]$. By using the Weierstrass preparation theorem in $k[[x^p,y]]$ it follows that 
\[
f(x^p,y)=u(x^p,y)[f_0(x^p)+f_1(x^p)y+\cdots + f_{m-1}(x^p)y^{m-1}+y^m],
\]
where $f_i(x^p)\in (x^p)$, for all $0\leq m-1$ and $u(x^p,y)$ is a unit in $k[[x^p,y]]$ and hence also in $k[[x,y]]$. In fact $m \geq 2$ since it assumed that $Q\in \tilde{C}$ is singular. Then $I_C=(y^m+h(x^p,y))$, where 
\[
h(x^p,y)=f_0(x^p)+f_1(x^p)y+\cdots + f_{m-1}(x^p)y^{m-1} \in (x,y)^{p+1}\subset k[[x,y]]
\]
and $m\geq 2$. Suppose that $m\geq p$. Then $m_P(C) \geq p$ and hence $p_a(C) \geq p$, which is impossible since by assumption 
$\mathrm{p}_a(C)\leq (p-1)/2$. Suppose that $m<p$.  Then write $p=sm+r$, $0<r<m$. After blowing up $P\in C$ and its infinitely near singular points $s$ times we see by using the adjunction formula that 
\begin{gather}\label{sec1-prop4-eq1}
2\mathrm{p}_a(C) \geq sm(m-1).
\end{gather}
Suppose that $m\geq (p+1)/2$. Then $m-1\geq (p-1)/2$ and hence  from the above inequality it follows that 
\[
\mathrm{p_a}(C)\geq \left(\frac{sm}{2}\right)\left(\frac{p-1}{2}\right)\geq \frac{p-1}{2},
\]
since $m\geq 2$. 

Suppose that $m < (p+1)/2$. Then also $r<m < (p+1)/2$. Then $p-r>(p-1)/2$ and hence
\begin{gather}\label{sec1-eq-111}
\mathrm{p}_a(C)\geq \frac{1}{2}sm(m-1)=(p-r)\frac{m-1}{2} \geq \left(\frac{p-1}{2}\right) \left(\frac{m-1}{2}\right).
\end{gather}
Suppose that $m\geq 3$. Then from the above inequality it follows that $\mathrm{p}_a(C)\geq (p-1)/2$. Suppose that $m=2$. Then  $s=(p-1)/2$ and $r=1$. Then from the equation~\ref{sec1-eq-111} it follows again that $\mathrm{p}_a(C)\geq (p-1)/2$.

\textbf{Case 2.} $Q\in \tilde{C}$ is smooth. Then  $C \rightarrow \tilde{C}$ is a $\mu_p$ or $\alpha_p$ torsor. Hence 
\[
\mathcal{O}_C=\frac{\mathcal{O}_{\tilde{C}}[t]}{(t^p-s)}
\]
where $s\in \mathcal{O}_{\tilde{C}}$. Let $x$ be local analytic coordinate of 
$\tilde{C}$ at $Q$. Then locally analytically at $Q\in \tilde{C}$, $\mathcal{O}_{\tilde{C}}=k[[x]]$ and $s=f(x)\in k[[x]]$. Moreover, since $P\in C$ is singular, $f(x)\in (x^2)$. Therefore
\[
\mathcal{O}_C=\frac{\mathcal{O}_{\tilde{C}}[t]}{(t^p-s)}=\frac{k[[x,t]]}{(t^p-f(x))}.
\]
Then one can write $f(x)=x^m u(x)$, where $u(x)$ is a unit in $k[[x]]$. If $m< p$ then $\sqrt[m]{u(x)}$ exists and therefore locally analytically at $P$, 
\[
\mathcal{O}_C\cong \frac{k[[x,y]]}{(t^p-x^m)}.
\]
If $p \leq m$ then since $k$ has characteristic $p$, the $\sqrt[m]{u(x)}$ does not always exist. But in this case $m_P(\mathcal{O}_{C,P})\geq p$ which is impossible since $\mathrm{p}_a(C)<p$. Hence $I_C=(t^p-x^m)$, $m\geq 2$. Then by using the same argument as in Case 1 it follows that $p_a(C)\geq (p-1)/2$, which is impossible.

Hence every singular point of $C$ is a fixed point of $D$. Hence $D$ lifts to a vector field $D^{\prime}$ on the blow up $X^{\prime}$ of $X$ at any singular point of $C$. Let $C^{\prime}$ be the birational transform of $C$ in $X^{\prime}$. Then $D^{\prime}(I_{C^{\prime}})\subset I_{C^{\prime}}$ and $p_a(C^{\prime})< p_a(C)$. Hence $D^{\prime}$ restricts to a vector field of $C^{\prime}$. Moreover, the previous arguments imply that the singular points of $C^{\prime}$ are fixed points of $D^{\prime}$. Hence this process can continue until a birational map $f \colon Y \rightarrow X$ is reached such that $Y$ and the birational transform $\bar{C}=f_{\ast}^{-1}C$ are smooth and $D$ lifts to a vector field $\bar{D}$ on $Y$ such that $\bar{D}(I_{\bar{C}})\subset I_{\bar{C}}$ and hence it induces a vector field on $\bar{C}$ lifting $D$.

\end{proof}

\begin{corollary}\label{sec1-cor2}
With assumptions as in Proposition~\ref{sec1-prop4}. Suppose in addition that $C$ is singular. Let $D_c$ be the vector field on $C$ induced by $D$. Suppose that $D_c\not=0$. Let $\bar{C}\rightarrow C$ be the normalization of $C$. Then $\bar{C}\cong \mathbb{P}^1_k$. Moreover
\begin{enumerate}
\item Suppose that $D^p=0$. Then $D$ has exactly one fixed point on $C$.
\item Suppose that $D^p=D$. Then $D$ has at most two distinct fixed points on $C$.
\end{enumerate}
In particular, $C$ is rational.

\end{corollary}

\begin{proof}
By Proposition~\ref{sec1-prop4}, $D$ fixes the singular points of $C$ and the restriction $D_c$ of $D$ on $C$ lifts to a vector field 
$\bar{D}$ on the  normalization $\pi \colon \bar{C} \rightarrow C$ of $C$. Considering that smooth curves of arithmetic genus greater or equal than 2 do not have nontrivial global vector fields, it follows that 
$\mathrm{p}_a(\bar{C})\leq 1$.  

Suppose that $\bar{C}$ is an elliptic curve. In this case $T_{\bar{C}}=\mathcal{O}_{\bar{C}}$ and hence the unique global vector field of $\bar{C}$ has no fixed points. Let $P\in C$ be a singular point of $C$. Then by Proposition~\ref{sec1-prop4}, $P$ is a fixed point of $D$. Let also  $\pi^{-1}(P)=\sum_{i=1}^n m_iQ_i$, be the divisor in $\bar{C}$  corresponding to $\pi^{-1}(P)$. Then since $p_a(C)<(p-1)/2$, it follows that $m_i<p$, for all $i=1,\ldots, m$. Then by Lemma~\ref{sec11-lemma-1} it follows that every $Q_i$, $i=1,\ldots, n$, is a fixed point of $\bar{D}$. This a contradiction since $\bar{D}$ has no fixed points.

 Hence $\bar{C}=\mathbb{P}^1$. In this case $T_{\bar{C}}=\omega_{\mathbb{P}^1}^{-1}=\mathcal{O}_{\mathbb{P}^1}(2)$. Hence $\mathbb{P}^1$ has three linearly independent global vector fields $D_i$, $i=1,2,3$. These vector fields are induced from the homogeneous vector fields $D_1=x\frac{\partial}{\partial x}$, $D_2=x\frac{\partial}{\partial y}$ and $D_3=y\frac{\partial}{\partial x}$ of $k[x,y]$. Note that $D_1^p=D_1$ and $D_i^p=0$, $i=2,3$. Hence there are $a_i \in k$, $i=1,2,3$, such that $\bar{D}=a_1D_1+a_2D_2+a_3D_3$. 

\textbf{Claim:} $\bar{D}^p=\bar{D}$ if and only if $a_2=a_3=0$ and $a_1\in\mathbb{F}_p^{\ast}$, and $\bar{D}^p=0$ if and only if $a_1^2+4a_2a_3=0$. 

In order to show this restrict $\bar{D}$ to the standard affine cover of $\mathbb{P}^1$.

Let $U \subset \mathbb{P}^1$ be the open affine subset given by $y\not= 0$. Let $u=x/y$. Then an easy calculation shows that $D_1=u\frac{d}{du}$, $D_2=-u^2\frac{d}{du}$ and $D_3=\frac{d}{du}$. Therefore 
\[
\bar{D}=(-a_2u^2+a_1u+a_3)\frac{d}{du}
\]
in $U$. I will now show that this is additive if and only if $-a_2u^2+a_1u+a_3=0$ has either a double root or no roots and multiplicative if and only if $a_2=0$ and $a_1\in \mathbb{F}_p$. Suppose that the previous equation  has a double root, and hence $a_1^2+4a_2a_3=0$. Then after a linear automorphism of $k[u]$, $\bar{D}=au^2\frac{d}{du}$, $a\in k$. This can easily verified to be additive. Suppose on the other hand that $-a_2u^2+a_1u+a_3=0$ has either two distinct roots or only one simple root (hence $a_2=0$). Suppose that $a_2\not=0$ and hence it has two distinct roots. Then after a linear automorphism of $k[u]$, $\bar{D}=a(u^2+u)\frac{d}{du}$. Then an easy calculation shows that
\[
D^p(u^{p-1})=a^p(p-1)^p(u^p+u^{p-1})=-a^p(u^p+u^{p-1})\not= 0.
\]
Hence in this case $\bar{D}$ is neither additive or multiplicative. Hence $a_2=0$ and $\bar{D}=(a_1u+a_3)\frac{d}{du}$. Then $\bar{D}^p=a_1^{p-1}\bar{D}$. Hence $\bar{D}^p=\bar{D}$ if and only if $a_1^{p-1}=1$ and therefore if and only if $a_1\in\mathbb{F}_p$. 

Let $V$ be the affine open subset of $\mathbb{P}^1$ given by $x\not=0$. Let $v=y/x$. Then in $V$, $D_1=- v\frac{d}{dv}$, $D_2=\frac{d}{dv}$ and $D_3=-v^2\frac{d}{dv}$. Therefore
\[
\bar{D}=(-a_3v^2-a_1v+a_2)\frac{d}{dv}.
\]
Suppose that $\bar{D}$ is additive. Then similar arguments as before show that $a_1^2+4a_2a_3 =0$. Suppose that $\bar{D}$ is of multiplicative type. Then as before we get that $a_3=0$. This concludes the proof of the claim.

Suppose now that $\bar{D}$ is of multiplicative type. Then it has been shown that $\bar{D}=ax\frac{\partial}{\partial x}$, $a\in \mathbb{F}_p^{\ast}$. The fixed points of this are $[0,1]$ and $[1,0]$. In particular it has exactly two distinct fixed points. These points may be over different  points of $C$ or over the same. Hence $D$ has at most 2 fixed points on $C$ as claimed.

Suppose that $\bar{D}$ is of additive type. Then from the previous arguments it follows that $\bar{D}$ has a single fixed point. 

Hence if $D^p=D$, then $D$ has at most two distinct points and if $D^p=0$ then it has just one.
\end{proof}

\begin{proposition}\label{sec1-prop-8}
Let $X$ be a canonically polarized  surface over an algebraically closed field of characteristic $p>0$. Let $D$ be a nonzero vector field on $X$ such that either $D^p=0$ or $D^p=D$. Assume moreover that $D$ fixes the singular points of $X$ and that it lifts to the minimal resolution of $X$. Suppose that $p>(m^2+3m)K_X^2+3$. Then the linear system $|mK_X|$ does not contain a positive dimensional subsystem  whose members are stabilized by $D$.
\end{proposition}
\begin{proof}

Suppose that there exists a positive dimensional linear subsystem of $|mK_X|$, for some $m>0$, whose members are stabilized by $D$. Then take  $|V|\subset |mK_X|$  a one-dimensional linear subsystem whose members are stabilized by $D$.

\textbf{Claim:} Let $C\in |V|$ be any member of $|V|$ and let $C=\sum_{i=1}^sn_iC_i$ be its decomposition into its reduced and irreducible components. Then, if  $C_i$ is not a component of the divisorial part of $D$, $C_i$ is a rational curve, for all $i=1,\ldots, s$.

Indeed. From the assumptions of the proposition it follows that $K_X \cdot C <p$. Then, since $K_X$ is ample, it follows by Corollary~\ref{sec1-cor1} that every $C_i$ is stabilized by $D$, i.e., $D(I_{C_i})\subset I_{C_i}$, $i=1,\ldots, n$. Hence $D$ induces vector fields on every $C_i$, for all $i$. 

Suppose that $C_i$ is a component of $C$ which is not contained in the divisorial part of $D$. Then the restriction of $D$ on $C_i$ is not zero. Let $\pi_i \colon \bar{C}_i \rightarrow C_i$ be the normalization of $C_i$. I will show next that $D$ lifts to  $\bar{C}_i$.  

Let $f \colon X^{\prime} \rightarrow X$ be the minimal resolution of $X$. Let $C_i^{\prime}$ be the birational transform of $C_i$ in $X^{\prime}$. Then $C_i^{\prime}$ is stabilized by $D^{\prime}$ and therefore $D^{\prime}$ induces a nonzero  vector field on $C_i^{\prime}$. In order to show that $D$ lifts to $\bar{C}_i$ it suffices to show that $D^{\prime}$ lifts to the normalization of $C_i^{\prime}$, which is $\bar{C}_i$. This will be done by using Proposition~\ref{sec1-prop4}.

Since $X$ has canonical singularities, $K_{X^{\prime}}=f^{\ast}K_X$. Then, since $K_X$ is ample, 
\begin{gather}\label{sec2-eq-1111}
K_{X^{\prime}}\cdot C_i^{\prime} =f^{\ast} K_X \cdot C^{\prime}_i=K_X\cdot C_i \leq K_X \cdot C=mK_X^2<\frac{p-3}{m+3} 
\end{gather}
by the assumptions of the proposition. Moreover, since $K_{X^{\prime}}$ is nef and big, by the Hodge Index Theorem and the previous inequality, it follows that
\begin{gather}\label{sec2-eq-1112}
(C_i^{\prime})^2\leq \frac{(K_{X^{\prime}}\cdot C_i^{\prime})^2}{K_{X^{\prime}}^2}\leq \frac{m^2(K_X^2)^2}{K_X^2}=m^2K_X^2\leq \frac{m}{m+3}(p-3).
\end{gather}
Now from the equations (\ref{sec2-eq-1111}), (\ref{sec2-eq-1112}) it follows that 
\begin{equation}\label{sec1-eq-100}
p_a(C^{\prime}_i)=1+\frac{1}{2}((C^{\prime}_i)^2+K^{\prime}_X\cdot C^{\prime}_i)<1+\frac{1}{2}\cdot \frac{m+1}{m+3}(p-3)<\frac{p-1}{2},
\end{equation}
Therefore, by Proposition~\ref{sec1-prop4}, $D^{\prime}$ lifts to the normalization of $C_i^{\prime}$ and hence $D$ lifts to a vector field $\bar{D}$ on the normalization $\bar{C}_i$ of $C_i$. Considering that a smooth curve of genus greater or equal to 2 does not have any nontrivial global vector fields, it follows that $\bar{C}_i$ is either $\mathbb{P}^1$ or an elliptic curve. I will show that it is actually $\mathbb{P}^1$. 

Next I will show that there exist fixed points of $D^{\prime}$ on $C_i^{\prime}$. 

Consider cases with respect to whether $C^{\prime}=f^{\ast}C$ is reducible or not.

Suppose that $C^{\prime}$ is irreducible (and hence $C$ does not pass through any singular point of $X$). Then 
$C^{\prime}=n_iC_i^{\prime}$. In particular, 
\begin{gather}\label{sec2-eq-1113}
m^2K_X^2=n_i^2(C_i^{\prime})^2=n_i^2C_i^2.
\end{gather}

Suppose that $D^{\prime}$ has no fixed points on $C_i^{\prime}$. Let $\pi \colon X^{\prime} \rightarrow Y^{\prime}$ be the quotient of $X^{\prime}$ by the $\alpha_p$ or $\mu_p$ action induced on $X^{\prime}$ by $D^{\prime}$. Let $\hat{C}_i=\pi(C_i^{\prime})$. Then, by~\cite{AA86}, $\hat{C}_i$ is in the smooth part of $Y^{\prime}$ and by~\cite{RS76}, $\pi^{\ast}\hat{C}_i=C^{\prime}_i$.  Hence
\[
(C^{\prime}_i)^2=(\pi^{\ast}\hat{C}_i)^2=p\hat{C}_i^2=\lambda p,
\]
for some $\lambda \in \mathbb{Z}$. Then from (\ref{sec2-eq-1113}) it follows that $m^2K_X^2=\lambda n_i^2p$. Since $K_X^2 >0$, then $\lambda >0$ and hence $K_X^2 >p$, which is a contradiction from the assumptions. Hence in this case there are fixed points of $D^{\prime}$ on $C_i^{\prime}$.

Suppose that $C^{\prime}$ has at least two components. Since $K_X$ is ample, $C$ and hence $C^{\prime}$ is connected. Hence $C_i^{\prime}$ intersects another component $B$ of $C^{\prime}$. If $B$ is contained in the divisorial part of $D^{\prime}$ then the intersection points $C_i^{\prime}\cap B$ are fixed points of $D^{\prime}$. Hence in this case there are fixed points of $D^{\prime}$ on $C_i^{\prime}$. Suppose that $B$ is not in the divisorial part of $D^{\prime}$. There are now two possibilities. $B$ is not $f$-exceptional or $B$ is $f$-exceptional. 

Suppose that $B$ is not $f$-exceptional. Then $B=C_j^{\prime}$, the birational transform in $X^{\prime}$ of a component $C_j$ of $C$, $j\not= i$. 
 But now from the equation (\ref{sec2-eq-1114}) in the proof of Proposition~\ref{sec1-prop3} it follows that
\[
C_i^{\prime} \cdot C_j^{\prime}\leq C_i \cdot C_j<p.
\]
Then from Corollary~\ref{sec1-cor0} it follows that the points of intersection $C_i^{\prime}\cap C_j^{\prime}$ are fixed points of $D^{\prime}$. Again then there are fixed points of $D^{\prime}$ on $C_i^{\prime}$.

Suppose finally that $B$ is $f$-exceptional. I will show that $B \cdot C_i^{\prime} < p$ and hence again from Corollary~\ref{sec1-cor0} the points of intersection $C_i^{\prime}\cap B$ are fixed points of $D^{\prime}$.

From the adjunction formula for $C_i^{\prime}$ it follows that
\[
(C_i^{\prime})^2\geq -2-K_{X^{\prime}}\cdot C_i^{\prime}=-2-K_X\cdot C_i \geq -2-\frac{1}{n_i}K_X \cdot C=-2-\frac{m}{n_i}K_X^2.
\]
Then
\[
C^{\prime}=f^{\ast} C=\sum_{r=1}^sn_sC_s^{\prime} +bB+E,
\]
where $b>0$ is an integer and $E$ is an effective $f$-exceptional divisor. Then
\[
m^2K_X^2\geq C\cdot C_i =f^{\ast}C \cdot C_i^{\prime}\geq n_i(C^{\prime}_i)^2 +bB\cdot C^{\prime}_i  \geq -2n_i -mK_X^2+b B\cdot C^{\prime}_i
\]
Therefore
\begin{gather}\label{sec2-eq-1115}
bB\cdot C^{\prime}_i \leq m^2K_X^2+2n_i+mK_X^2
\end{gather}
Now since $C\in |mK_X|$ and $K_X$ is ample it follows that $n_i \leq mK_X^2$. Then the previous equation becomes
\[
bB\cdot C^{\prime}_i \leq m^2K_X^2+2mK_X^2+mK_X^2=(m^2+3m)K_X^2<p,
\]
by the assumptions. Hence $b B\cdot C^{\prime}_i <p$, and in particular $B \cdot C_i^{\prime} <p$. Therefore, from Corollary~\ref{sec1-cor0} the points of intersection $C_i^{\prime}\cap B$ are fixed points of $D^{\prime}$. 

Therefore there are fixed points of $D^{\prime}$ on $C_i^{\prime}$. Let $P \in C_i^{\prime}$ be a fixed point of $D^{\prime}$. Let $\pi^{-1}(P) =\sum_{i=1}^m n_iQ_i$. Then since $p_a(C^{\prime}_i)<p$, it follows that $n_i<p$, $i=1,\ldots ,m$. Then by Lemma~\ref{sec11-lemma-1}, every $Q_i$ is a fixed point of $\bar{D}_i$. Hence $\bar{D}_i$ has fixed points. Therefore $\bar{C}_i\cong \mathbb{P}^1$ since vector fields on an elliptic curve do not have fixed points. This concludes the proof of the claim.

Let $|V^{\prime}|$ be the linear system which is obtained from $|V|$ by removing the base components. Hence $|V^{\prime}|$ has only isolated base points. Let  $\phi \colon X \dasharrow \mathbb{P}^1$ be the rational map defined by $|V^{\prime}|$. Consider now the following commutative diagram
\[
\xymatrix{
W \ar[r]^{h}\ar[dr]^{\psi}\ar[d]^{g} & B\ar[d]^{\sigma} \\
X \ar@{-->}[r]^{\phi}  & \mathbb{P}^1
}
\]
Where $g$ is the resolution of base points of $|V^{\prime}|$, $\psi$ the corresponding morphism, and $h$, $\sigma$ is the Stein factorization of $\psi$. Then $h$ is a fibration and its generic fiber is an integral normal (and hence regular) curve~\cite[Page 91]{Ba01}. Moreover, by the construction of $h$, the general fiber is the birational transform in $W$ of an irreducible component of a general member of $|V^{\prime}|$. Therefore it is a rational curve. 

Suppose that the general fiber of $h$  is smooth. Therefore the general fiber of $h$ is isomorphic to $\mathbb{P}^1$. Then the generic fiber is also a smooth curve of genus zero over $K(B)$, where $K(B)$ is the function field of $B$. Hence it is isomorphic to a smooth conic in $\mathbb{P}^2_{K(B)}$. Then by Tsen's Theorem this conic has a $K(B)$-point and therefore the generic fiber is actually isomorphic to 
$\mathbb{P}^1_{K(B)}$. Therefore, $X$, and hence $X^{\prime}$, is birational to $B \times \mathbb{P}^1$, i.e., is birationally ruled. But then this implies  that $X^{\prime}$ has Kodaira dimension $-1$, which is a contradiction. 

Hence every fiber of $h$ is singular and therefore the generic fiber is singular too. Then by Tate's Theorem~\cite{Ta52},~\cite{Sch09}, 
$(p-1)/2 < p_a(W_g)$, where $W_g$ is the general fiber of $h$. But since the general  fiber of $h$ is the birational transform of a component $C_i$ of a general member $C$ of $|V^{\prime}|$, it follows from the equation~\ref{sec1-eq-100} that
\[
p_a(W_g)\leq p_a(C)\leq 1+\frac{1}{2}(m+m^2)K_X^2 <(p-1)/2,
\]
a contradiction. Hence $|mK_X|$ contains at most finitely many integral curves of $D$. 

\end{proof}

\begin{corollary}\label{sec1-cor-8}
Let $X$ be a canonically polarized surface over an algebraically closed  field of characteristic $p>0$. Let $D$ be a nontrivial global vector field on $X$ such that $D^p=0$ or $D^p=D$. Suppose that 
\begin{enumerate}
\item $p>\mathrm{max}\{56,m^2+3m+3\}$, if $K_X^2=1$,
\item $
p>\mathrm{max}\{12\chi(\mathcal{O}_X)+11K_X^2+1,(m^2+3m)K_X^2+3\}$,
if $K_X^2 \geq 2$.
\end{enumerate}
Then the linear system $|mK_X|$ does not contain a positive dimensional subsystem  whose members are stabilized by $D$. 

Moreover, suppose that $D$ has only isolated singularities. Let $\pi \colon X \rightarrow Y$ be the quotient of $X$ by the $\alpha_p$ or $\mu_p$ action induced by $D$.  Then 
$h^0(\mathcal{O}_Y(mK_Y))\leq 1$.
\end{corollary}

\begin{proof}
From Theorem~\ref{sec11-th-2} it follows that $D$ lifts to the minimal resolution of $X$. Then from Proposition~\ref{sec1-prop-8} it follows that $|mK_X|$ does not contain a positive dimensional subsystem  whose members are stabilized by $D$.

Suppose now that $D$ has only isolated singularities. Then $K_X=\pi^{\ast}K_Y$. If $h^0(\mathcal{O}_Y(mK_Y))\geq 2$, then $|\pi^{\ast}(mK_Y)|$ gives a positive dimensional subsystem of $|mK_X|$ which consists of integral curves of $D$. But by Proposition~\ref{sec1-prop-8} this is impossible.
\end{proof}

The next two results will also be needed in the proofs of the main results of this paper.


\begin{proposition}\label{sec1-prop-6}
Let $f \colon Y \rightarrow X$ be a composition of $n$ blow ups starting from a smooth point $P \in X$ of a surface $X$. Let $C \subset X$ be an integral curve in $X$ passing through $P$ and let $m=m_Q(C)$ be the multiplicity of $C$ at $P\in C$. Then
\[
mK_Y-f^{\ast}C+C^{\prime}=mf^{\ast}K_X+\sum_{k=1}^n(km-a_1-a_2-\ldots -a_k)E_k,
\]
where $E_i$, $1\leq i \leq n$ are the $f$-exceptional curves, $C^{\prime}$ is the birational transform of $C$ in $Y$  and $0\leq a_i \leq m$, are nonnegative integers.
\end{proposition}
The proof of the proposition is by a simple induction on the number of blow ups $n$ and is omitted.

\begin{proposition}\label{sec1-prop-7}
Let $P\in S$ be a Duval singularity and let $C\subset S$ be a smooth curve such that $P\in S$. Let $f \colon S^{\prime}\rightarrow S$ be the minimal resolution of $P\in S$,  and $E_i$, $i=1,\dots, n$ be the $f$-exceptional curves. Let $C^{\prime}$ be the birational transform of $C$ in $S^{\prime}$ and $a_i>0$, $1\leq i \leq n$ be positive rational numbers such that
\[
f^{\ast}C=C^{\prime}+\sum_{i=1}^na_i E_i.
\]
Then
\begin{enumerate}
\item Suppose that $P\in S$ is of type $A_n$. Then $(n+1)C$ is Cartier in $S$ and $(n+1)a_i$ are positive integers $\leq n$, $i=1,\dots, n$.
\item Suppose that $P\in S$ is of type $D_n$. Then $4C$ is Cartier in $S$ and $4a_i$are integers $\leq n$, $i=1,\ldots, n$.
\item Suppose that $P\in S$ is of type $E_6$. Then $3C$ is Cartier in $S$ and $3a_i$are integers $\leq 6$, $i=1,\ldots, 6$.
\item Suppose that $P\in S$ is of type $E_7$. Then $2C$ is Cartier in $S$ and $2a_i$are integers $\leq 7$, $i=1,\ldots, 7$.
\end{enumerate}
\end{proposition} 
Notice that $P\in S$ cannot be of type $E_8$ because this singularity is factorial and hence there is no smooth curve passing through it.

The proof of this proposition is by a straightforward computation of the coefficients $a_i$ in $f^{\ast}C$ depending on the type of the singularity and the position of $C^{\prime}$ in the dual graph of the exceptional locus of the singularity and it is omitted. Similar computations can be found in~\cite[Proposition 4.5]{Tz03}.


\section{Methodology of the proof of Theorems~\ref{intro-the-1},~\ref{intro-the-2}.}\label{sec-3}

Let $X$ be a canonically polarized surface defined over an algebraically closed field of characteristic $p>0$ with a nontrivial global vector field $D$. The strategy for the proof of Theorems~\ref{intro-the-1},~\ref{intro-the-2} is to do one of the following:
\begin{enumerate}
\item Find an integral curve $C$ of $D$ on $X$ with the following properties: Its arithmetic genus $\mathrm{p}_a(C)$ is a function of $K_X^2$, $\mathrm{p}_a(\bar{C})\geq 1$, where $\bar{C}$ is the normalization of $C$, and such that $C$ contains some of the fixed points of $D$. Then by using the results of Section~\ref{sec-2}, if $\mathrm{p}_a(C)$ is small enough compared to the characteristic $p$, $D$ induces a vector field  on $C$ which lifts to $\bar{C}$. But this would be impossible since smooth curves of genus greater or equal than two have no nontrivial global vector fields and global vector fields on smooth elliptic curves do not have fixed points. This argument will allow us to conclude that if $p>f(K_X^2)$, for some function $f(K_X^2)$ of $K_X^2$ then $X$ does not have any nontrivial global vector fields.
\item Find a positive dimensional family of integral curves $\{C_t\}$ of $D$ whose arithmetic genus is a function of $K_X^2$ and $\chi(\mathcal{O}_X)$.  Then from Corollary~\ref{sec1-cor-8} there must be a relation of the form $p<f(K_X^2,\chi(\mathcal{O}_X))$. Hence if such a  relation does not hold, $X$ does not have any nontrivial global vector fields.

\end{enumerate}

In order to achieve this, the following method will be used. It is based on a method initially used in~\cite{RS76} and then in~\cite{Tz17} but with different objectives.

Since $X$ has a nontrivial global vector field, then by~\cite[Proposition 4.1]{Tz17} $X$ has a nontrivial global vector field $D$ of either additive or multiplicative type which induces a  nontrivial $\alpha_p$ or 
$\mu_p$ action.  Let $\pi \colon X \rightarrow Y$ be the quotient. Then $\pi$ is purely inseparable of degree $p$, $Y$ is normal and $K_Y$ is $\mathbb{Q}$-Cartier. Consider now the following diagram

\begin{equation}\label{sec2-diagram-1}
\xymatrix{
    Y^{\prime}\ar[d]_h \ar[dr]^g    &         X \ar[d]^{\pi} & \\
 Z                                                 & Y   
}
\end{equation}
where $g \colon Y^{\prime} \rightarrow Y$ is the minimal resolution of $Y$ and  $h \colon Y^{\prime} \rightarrow  Z$ its minimal model.

\begin{lemma}\label{sec3-lema-1}
Every $g$-exceptional curve is a rational curve (perhaps singular).
\end{lemma}
\begin{proof}
let $\hat{X}$ be the normalization of $Y^{\prime}$ in $K(X)$. Let $\phi \colon W \rightarrow \hat{X}$ be the minimal resolution of $\hat{X}$. Then there exists a commutative diagram
\[
\xymatrix{
W \ar[r]^{\phi}\ar[d]^{\psi} & \hat{X} \ar[r]^{\hat{\pi}}\ar[d]^{\hat{g}}&Y^{\prime} \ar[d]^g \\
 X^{\prime} \ar[r]^f &X \ar[r]^{\pi} & Y 
}
\]
where $\hat{\pi}$ is purely inseparable of degree $p$, $f\colon X^{\prime} \rightarrow X$ is the minimal resolution of $X$ and $\psi$ is birational.  Considering that $X$ has rational double points, the $f$ exceptional curves are smooth rational curves. Therefore, since $\psi$ is a composition of blow ups,  it easily follows that every $\hat{g}$-exceptional curve is a rational curve. Now let $F$ be a $g$-exceptional curve. Then $F=\hat{\pi}(\hat{F})$, where $\hat{F}$ is a $\hat{g}$-exceptional curve. Hence, $F$ is a rational curve.

\end{proof}

Integral curves on $X$ will be found by choosing a suitable a reflexive sheaf $L$ on $Y$ such that either $h^0(L)\geq 2$, in which case the pullbacks in $X$ of the divisors of $Y$ corresponding to the sections of $L$ will be integral curves of $D$, or $h^0((\pi^{\ast}L) ^{[1]})\geq 2$ and then study the action of $D$ on $H^0((\pi^{\ast}L)^{[1]})$ exhibited in  Proposition~\ref{sec1-prop1}. The eigenvectors of this action will be curves stabilized  by $D$ and under suitable conditions their components which are not contained in the divisorial part of $D$ will be integral curves of $D$.

In order to prove  Theorems~\ref{intro-the-1},~\ref{intro-the-2} we will distinguish cases with respect to the Kodaira dimension $\kappa(Z)$ of $Z$. Then results from the classification of surfaces in positive characteristic will be heavily used~\cite{BM76},~\cite{BM77},~\cite{Ek88} and the geometry o $X$ and $Z$ will be compared by using the diagram (\ref{sec2-diagram-1}). Moreover, since $\pi$ is a purely inseparable map, it induces an equivalence between the \'etale sites of $X$ and $Y$. Therefore $X$ and $Y$ have the same algebraic fundamental group, $l$-adic betti numbers and \'etale Euler characteristic. Then by using the fact that $g$ and $h$ are  birational it will be possible to calculate the algebraic fundamental group, $l$-adic Betti numbers and \'etale Euler characteristic of $X$ from those of $Z$.

The proof of Theorems~\ref{intro-the-1},~\ref{intro-the-2} is significantly easier if the vector field $D$ has a nontrivial divisorial part as the next theorem shows.

\begin{theorem}\label{sec3-lemma-2}~\cite[Theorem 6.1]{Tz17}
Suppose that $D$ has a nontrivial divisorial part. Suppose that $K_X^2<p$. Then the Kodaira dimension of $Z$ is $-1$ and  $X$ is purely inseparably uniruled.
\end{theorem}

Finally I collect some formulas and set up some terminology and notation that will be needed in the proofs.

 Let $\Delta$ be the divisorial part of $D$. There is also the following  adjunction formula for purely inseparable maps~\cite[Corollary 1]{RS76} 
\begin{gather}\label{sec2-eq-2}
K_X=\pi^{\ast}K_Y+(p-1)\Delta.
\end{gather}
(According to~\cite{RS76}, the previous formula holds in the smooth part of $X$ and hence everywhere since $X$ is normal).
 
Let $F_i$, $i=1,\ldots, n$ be the $g$-exceptional curves and $E_j$, $j=1,\ldots,m$ be the $h$-exceptional curves. By Lemma~\ref{sec3-lema-1} the $g$-exceptional curves $F_i$ are all rational (but perhaps singular). 

Taking into consideration  that  $g \colon Y^{\prime}\rightarrow Y$ is the minimal resolution of $Y$, we get the following adjunction formulas
\begin{gather}\label{sec2-eq-1}
K_{Y^{\prime}}+\sum_{i=1}^na_iF_i=g^{\ast}K_Y,\\
K_{Y^{\prime}}=h^{\ast}K_Z+\sum_{j=1}^mb_jE_j,\nonumber
\end{gather}
where $a_i \in\mathbb{Q}_{\geq 0}$, and $b_j\in \mathbb{Z}_{>0}$, $j=1,\ldots m$. Moreover since both $Y^{\prime}$ and $Z$ are smooth, $h$ is the composition of $m$ blow ups.


In the next sections I will consider cases with respect to the Kodaira dimension $\kappa(Z)$ of $Z$. 

Finally, for the rest of the paper, fix the notation of this section.

\section{The Kodaira dimension of $Z$ is 1 or 2.}\label{sec4}
\begin{proposition}\label{sec4-prop}
Let $X$ be a canonically polarized surface over a field of characteristic $p>0$. Suppose that $X$ has a nontrivial global vector field $D$ with isolated singularities such that $D^p=0$ or $D^p=D$. Suppose moreover, with notation as in Section~\ref{sec-3}, that the Kodaira dimension $\kappa(Z)$ of $Z$ is 1 or 2. then
\begin{enumerate}
\item Suppose that $K_X^2=1$. Then $p< 56$.
\item Suppose that $K_X^2\geq 2$. Then $
p<42K_X^2+3.$
\end{enumerate}

\end{proposition}
\begin{proof}
Suppose that the statements of the proposition are not true, i.e., $p\geq 56$, if $K_X^2=1$ and that $p>42K_X^2+3$,
 if $K_X^2\geq 2$. Then, I will sow that also $p>12\chi(\mathcal{O}_X)+11K_X^2+1$ and therefore from Theorem~\ref{sec11-th-2}  $D$ fixes the singular points of $X$ and lifts to a vector field $D^{\prime}$ in the minimal resolution of $f$.
 
 Indeed. Let $f\colon X^{\prime} \rightarrow X$ be the minimal resolution of $X$. Since $X$ has canonical singularities, $\chi(\mathcal{O}_X)=\chi(\mathcal{O}_{X^{\prime}})$ and $K_{X^{\prime}}=f^{\ast}K_X$. $X^{\prime}$ is a minimal surface of general type. Therefore from Noether's inequality, $2\chi(\mathcal{O}_{X^{\prime}})\leq K_{X^{\prime}}^2+6$. Hence, since $K_X^2=K_{X^{\prime}}^2$, it follows that $2\chi(\mathcal{O}_X)\leq K_X^2+6$.  Hence
 \begin{gather}\label{00000}
 12\chi(\mathcal{O}_X)+11K_X^2+1\leq 17K_X^2+37<42K_X^2+3<p,
 \end{gather}
 by the assumption.  
 
 Now by Theorem~\ref{sec3-lemma-2}, $D$ has no divisorial part, i.e., $\Delta=0$. Therefore, $K_X=\pi^{\ast}K_Y$ and hence $K_Y$ is ample.

Consider cases with respect to the Kodaira dimension $\kappa(Z)$ of $Z$. 

\textbf{Case 1: Suppose that $\kappa(Z)=2$.}

According to~\cite[Theorem 1.20]{Ek88}, the linear system $|4K_Z|$ is very ample. Let $W\in|4K_Z|$ be a smooth member which does not go through the points blown up by $h$ in the diagram~\ref{sec2-diagram-1}. Then by the adjunction formula, $p_a(W)=10K_Z^2+1$. Then combining the equations~\ref{sec2-eq-1} it follows that
\begin{gather}\label{sec3-eq-1}
g^{\ast}(4K_Y)=4K_{Y^{\prime}}+4\sum_{i=1}^na_iF_i=h^{\ast}(4K_Z)+4\sum_{j=1}^mb_jE_j+4\sum_{i=1}^na_iF_i\sim \\
W^{\prime}+4\sum_{j=1}^mb_jE_j+4\sum_{i=1}^na_iF_i,
\end{gather}
where $W^{\prime}=h^{\ast}W=h_{\ast}^{-1}W$ is the birational transform of $W$ in $Y^{\prime}$. By pushing down to $Y$ we get that 
\begin{gather}\label{sec3-eq-2}
4K_Y\sim \tilde{W}+4\sum_{i=1}^m b_i\tilde{E}_i,
\end{gather}
where $\tilde{E}_i=g_{\ast}E_i$, $1\leq i \leq m$. Note that since $Y^{\prime}$ is the minimal resolution of $Y$, $g$ does not contract any  (-1) $h$-exceptional curves. Hence if $h$ is not an isomorphism then $g_{\ast}\sum_{i=1}^mE_i \not= 0$. Now since $|4K_Z|$ is very ample it follows that $\dim |\tilde{W}|\geq 1$ and therefore $\dim |4K_Y|\geq 1$, or equivalently $h^0(\mathcal{O}_Y(4K_Y))\geq 2$. But by Corollary~\ref{sec1-cor-8} this is impossible.

\textbf{Case 2: Suppose that $\kappa(Z)=1$.} 

Since $\kappa(Z)=1$, it is well known that $Z$ admits an elliptic fibration $\phi \colon Z \rightarrow B$, where $B$ is a smooth curve. Then one can write
\begin{gather}\label{sec3-eq-4}
R^1\phi_{\ast}\mathcal{O}_Z=L\oplus T,
\end{gather}
where $L$ is an invertible sheaf on $B$ and $T$ is a torsion sheaf. 

\textbf{Claim:} $B\cong \mathbb{P}^1$ and  $T=0$. 

By Lemma~\ref{sec3-lema-1}, the $g$-exceptional curves are rational. Hence if at least one of them is not contracted to a point by $\phi \circ h$, then $B$ is dominated by a rational curve and  hence it is isomorphic to $\mathbb{P}^1$. Suppose that every $g$-exceptional curve is contracted to a point by $\phi \circ h$. Then by looking at diagram~\ref{sec2-diagram-1} we see that there exists factorizations
\[
\xymatrix{
    & Y \ar[dr]^{\psi} & \\
    X \ar[ur]^{\pi} \ar[rr]^{\sigma} & & B
    }
\]
such that the general fiber of $\psi$ is an elliptic curve. Then let $Y_b=\psi^{-1}(b)$ be the general fiber. Then $K_Y\cdot Y_b=0$ and therefore, 
\[
K_X\cdot \pi^{\ast} Y_b=\pi^{\ast}K_Y \cdot \pi^{\ast}Y_b=p K_Y \cdot Y_b=0.
\]
But this is impossible since $K_X$ is ample. Therefore there must be a $g$-exceptional curve not contracted to a point by $\phi\circ h$ and hence $B\cong \mathbb{P}^1$.

Suppose now that $T\not= 0$. Let  $b \in T$. Then $\phi^{-1}(b)=pmW$, $m>0$ and $W$ is an idecomposable fiber~\cite{KU85}. Moreover $|14KZ|$ defines the fibration $\phi$~\cite{KU85}. Hence $14K_Z \sim \nu F$, where $F$ is a general fiber of $\phi$ and hence a smooth elliptic curve (if $p\not= 2,3$.).  Then 
$F\sim \phi^{-1}(b)=pmW$. and hence $14K_Z\sim pm\nu W$. Then by pulling up to $Y^{\prime}$ it follows that
\[
14h^{\ast}K_Z=pm\nu W^{\prime} +p(\sum_{i=1}^m c_i E_i).
\]
If $h$ blows up a point of $W$ then $c_i>0$ and $14h^{\ast}K_Z$ has a component corresponding to a $(-1)$ $h$-exceptional curve with coefficient divisible by $p$. Considering that the $(-1)$ $h$-exceptional curves do not contract by $g$, we see that in any case (if $h$ blows up a point on $W$ or not) that, after pushing down to $Y$, $14K_Y \sim p \tilde{W} + B $, for some divisor $\tilde{W}$ (either the birational transform of $W$ or the image of a $-1$ $h$-exceptional curve.  Therefore by pulling up to $X$ and since $K_X=\pi^{\ast} K_Y$,
\[
14K_X \sim p\pi^{\ast}\tilde{W} +\pi^{\ast}B.
\]
But from this it follows that $14K_X^2>p$, a contradiction. This concludes the proof of the claim.

Next consider cases with respect to $p_g(Z)$.

\textbf{Case 1.} Suppose that $p_g(Z)\geq 2$. Then, since $h^0(\mathcal{O}_Z(K_Z))\geq 2$, it easily follows that $h^0(\mathcal{O}_Y(K_Y))\geq 2$. Then by Corollary~\ref{sec1-cor-8} we get a contradiction. So this case is impossible too.

\textbf{Case 2.} Suppose that $\mathrm{p}_g(Z) \leq 1$. I will show that this case is impossible too.

From the Noether's formula on $Z$~\cite[Theorem 5.1]{Ba01}
\begin{gather}\label{sec3-eq-6}
10-8h^1(\mathcal{O}_Z)+12p_g(Z)=K_Z^2+b_2(Z)+2(2h^1(\mathcal{O}_Z)-b_1(Z))=\\
b_2(Z)+2(2h^1(\mathcal{O}_Z)-b_1(Z))\nonumber
\end{gather}
it easily follows~\cite[Page 113]{Ba01} that if $p_g(Z) \leq 1$, then the only numerical solutions to the equation~\ref{sec3-eq-6} are the following:
\begin{enumerate}
\item $p_g(Z)=0$, $\chi(\mathcal{O}_Z)=0$, $b_1(Z)=2$.
\item $p_g(Z)=0$, $\chi(\mathcal{O}_Z)=1$, $b_1(Z)=0$.
\item $p_g(Z)=1$, $\chi(\mathcal{O}_Z)=2$, $b_1(Z)=0$.
\item $p_g(Z)=1$, $\chi(\mathcal{O}_Z)=1$, $b_1(Z)=2$.
\item $p_g(Z)=1$, $\chi(\mathcal{O}_Z)=1$, $b_1(Z)=0$.
\item $p_g(Z)=1$, $\chi(\mathcal{O}_Z)=0$, $b_1(Z)=2$.
\item $p_g(Z)=1$, $\chi(\mathcal{O}_Z)=0$, $b_1(Z)=4$.
\end{enumerate}
Note that by~\cite[Lemma 3.5]{KU85} the last case is not possible. Consider next each one of the cases separately. I will only consider the first two cases. The rest are similar and are omitted.

\textbf{Case 2.1.} Suppose that $p_g(Z)=\chi(\mathcal{O}_Z)=0$ and $b_1(Z)=2$.

By Igusa's formula~\cite{IG60} it follows that the fibers of $\phi \colon Z \rightarrow \mathbb{P}^1$ are either smooth elliptic curves or of  type $mE$, where $m$ is a positive integer and $E$ an elliptic curve (singular or smooth). Also note that $\phi$ must have multiple fibers or else $Z$ cannot have Kodaira dimension 1. 

I will next show that in fact $E$ is a smooth elliptic curve. Indeed. Since $b_1(Z)=2$ it follows that $\dim \mathrm{Alb}(Z)=1$. Hence $\mathrm{Alb}(Z)$ is a smooth elliptic curve. Let then $\psi \colon Z \rightarrow \mathrm{Alb}(Z)$ be the Albanese map. Then there exist the following two maps
\[
\xymatrix{
Z\ar[r]^{\psi}\ar[d]^{\phi} & \mathrm{Alb}(Z)\\
\mathbb{P}^1 & \\
}
\]
Suppose that $mE$ is a multiple fiber of $\phi$. Suppose also that $E$ is a rational elliptic curve. Then $E$ cannot dominate $\mathrm{Alb}(Z)$ and hence it must contract by $\psi$. Hence all fibers of $\phi$ contract by $\psi$. But then there would be a nontrivial map $\mathbb{P}^1 \rightarrow \mathrm{Alb}(Z)$, which is impossible. Hence $E$ is a smooth elliptic curve. 

It is well known~\cite[Theorem 8.11]{Ba01} that the linear system $|\nu K_Z|$, $\nu\in \{4,6\}$ contains a strictly positive divisor. Then $\nu K_Z\sim sE$, where $s>0$ is a positive integer and $E$ is a smooth elliptic curve. Let $E^{\prime}=h_{\ast}^{-1}E$ be the birational transform of $E$ in $Y^{\prime}$. Then $E^{\prime}$ is a smooth elliptic curve and since the $g$-exceptional curves are all rational, it follows that $E^{\prime}$ does not contract by $g$. Therefore by pulling up to $Y^{\prime}$ and then pushing down to $Y$ we get that 
\begin{gather}\label{sec3-eq-1111}
\nu K_Y \sim m\tilde{E} +B,
\end{gather}
where $B$ is an effective divisor on $Y$. Hence by pulling up to $X$ we get that
\begin{gather}\label{sec3-eq-4444}
\nu K_X \sim m\hat{E} +\pi^{\ast} B.
\end{gather}
As in the previous cases we see that if $K_X^2< p/\nu$, $\hat{E}$ is irreducible and therefore is an integral curve of $D$ whose normalization $\bar{E}$ is a smooth elliptic curve. I will show that $D$ lifts to a vector field $\bar{D}$ on $\bar{E}$ and that $D$ has fixed points on $\hat{E}$. Then by Lemma~\ref{sec11-lemma-1}, $\bar{D}$ will have fixed points which is impossible since $\bar{E}$ is an elliptic curve and hence get a contradiction again.

Let now $f\colon X^{\prime} \rightarrow X$ be the minimal resolution of $X$. Then $K_{X^{\prime}}=f^{\ast}K_X$ and therefore
\begin{gather}\label{sec3-eq-7}
\nu K_{X^{\prime}} \sim mE^{\prime\prime} +f^{\ast}\pi^{\ast} B+F,
\end{gather}
where $E^{\prime\prime}$ is the birational transform of $\hat{E}$ in $X^{\prime}$ and $F$ is an effective $f$-exceptional divisor. Now from the equation (\ref{sec3-eq-7}), since $K_{X^{\prime}}$ is nef and big,  we get that 
\begin{gather}\label{sec3-eq-1113}
K_{X^{\prime}} \cdot E^{\prime\prime} < \nu K_{X^{\prime}}^2=\nu K_X^2.
\end{gather}
 and then from the Hodge Index Theorem it follows that that 
 \[
(E^{\prime\prime})^2<\frac{(K_{X^{\prime}}\cdot E^{\prime\prime})^2}{K_{X^{\prime}}^2} <\nu^2 K_X^2.
\]
 Therefore from the adjunction formula it follows that
\[
\mathrm{p}_a(E^{\prime\prime}) < \frac{{\nu}(\nu+1)}{2}K_X^2 +1.
\]
 Hence if 
 \[
 K_X^2< \frac{p-3}{2}\cdot \frac{2}{\nu(\nu+1)},
 \]
 then $\mathrm{p}_a(E^{\prime\prime}) <(p-1)/2$. Considering that $\nu \in\{4,6\}$, the above inequality holds if $K_X^2<(p-3)/42$,  which holds according by the assumptions. Also, since $\hat{E}$ is an integral curve of $D$, $E^{\prime\prime}$ is an integral curve of $D^{\prime}$, the lifting of $D$ to $X^{\prime}$. Therefore in this case, from Proposition~\ref{sec1-prop4} it follows that 
 the restriction  of $D^{\prime}$ on $E^{\prime\prime}$ fixes the singular points of $E^{\prime\prime}$ and hence lifts to its normalization $\bar{E}$ of $E^{\prime\prime}$. 
 
 Next I will show that $D^{\prime}$ has fixed points on $E^{\prime\prime}$. 
 
 Suppose that $D$ has no fixed points on $\hat{E}$. Then $\hat{E}$ is in the smooth part of $X$ since the singular points of $X$ are fixed points of $D$. Moreover, since $D$ has no fixed points on $\hat{E}$, $\tilde{E}=\pi(\hat{E})$ is in the smooth part of $Y$. Then 
 \[
 K_X\cdot \hat{E}=\pi^{\ast}K_Y \cdot \pi^{\ast}\tilde{E} =p (K_Y \cdot \tilde{E}) =\lambda p,
 \]
 where, since $K_Y$ is ample, $\lambda$ is a positive integer. But then from  the equation (\ref{sec3-eq-1111}) it follows that $\nu K_X^2 >p$, which is impossible. Therefore, there are fixed points of $D$ on $\hat{E}$. Let $P\in \hat{E}$ be a point which is a fixed point of $D$. Suppose that $P\in X$ is a smooth point. Then $Q=f^{-1}(P) $ is a fixed point of $E^{\prime\prime}$. Suppose that $P\in X$ is singular. Let then $F$ be an $f$-exceptional curve such that $F \cdot E^{\prime\prime}>0$. By Theorem~\ref{sec11-th-2}, $F$ is an integral curve of $D^{\prime}$. I will show that $F \cdot E^{\prime\prime} <p$ and hence by Corollary~\ref{sec1-cor0}, the intersection points $F \cap E^{\prime\prime}$ are fixed points of $D^{\prime}$. Write
 \[
 f^{\ast}\hat{E} =E^{\prime\prime} +aF +F^{\prime},
 \]
 where $F^{\prime}$ is $f$-exceptional and effective. Then by Lemma~\ref{hodge}, and (\ref{sec3-eq-4444}), it follows that 
 $\hat{E}^2<\nu^2K_X^2$ and hence
 \[
 \nu^2K_X^2>\hat{E}^2\geq (E^{\prime\prime})^2+a(F\cdot E^{\prime\prime}).
 \]
 Considering now that from (\ref{sec3-eq-1113}), 
 \[
 (E^{\prime\prime})^2\geq -2-K_{X^{\prime}}\cdot E^{\prime\prime}\geq -2 -\nu K_X^2
 \]
 We get that
 \begin{gather}\label{sec3-eq-111111}
 a(F\cdot E^{\prime\prime}) \leq 2+(\nu+\nu^2)K_X^2,
 \end{gather}
 and therefore $F\cdot E^{\prime\prime} <p$ if $2+(\nu+\nu^2)K_X^2<p$, which holds if $42K_X^2+2<p$ ($\nu =4$ or $\nu=6$). Hence the intersection points $E^{\prime\prime} \cap F$ are fixed points of $D^{\prime}$. Hence in any case there are fixed points of $D^{\prime}$ on $E^{\prime\prime}$.  Then  from Lemma~\ref{sec11-lemma-1}, the preimages of these points in $\bar{D}$ are fixed points of the lifting of $D^{\prime}$ on $\bar{D}$, which is a contradiction since a vector field on an elliptic curve has no fixed points.

\textbf{Case 2.2.} Suppose that $p_g(Z)=0$, $\chi(\mathcal{O}_Z)=1$, $b_1(Z)=0$. I will show that this case is also impossible.

\textbf{Claim:} $\dim |6K_Z| \geq 1$.

Let $F_{t_i}=m_iP_i$, $t_i\in \mathbb{P}^1$, $i=1,\ldots , r$ be the multiple fibers of $\phi$. Since $T=0$, they are all tame. Then by the canonical bundle 
formula~\cite[Theorem 7.15 and Page 118]{Ba01} we get that
\begin{gather}\label{sec3-eq-8}
\dim |nK_Z|=n(-2+\chi(\mathcal{O}_Z))+\sum_{i=1}^r \left[ \frac{n(m_i-1)}{m_i}\right]=-n+\sum_{i=1}^r \left[ \frac{n(m_i-1)}{m_i}\right],
\end{gather}
where for any $m\in \mathbb{N}$, $[m]$ denotes its integer part. Also, in the notation~\cite[Remark 8.3]{Ba01} if, 
\[
\lambda(\phi)=-1+\sum_{i=1}^r\frac{m_i-1}{m_i},
\]
Then $\kappa(Z)=1$ if and only if $\lambda(\phi)>0$. Hence $\phi$ has at least two multiple fibers.

Suppose that $\phi$ has at least three multiple fibers, i.e., $r\geq 3$ and $m_i\geq 2$. Then for every $1 \leq i \leq r$, 
\[
\left[ 6(1-\frac{1}{m_i})\right] \geq \left[\frac{6}{2}\right] =3.
\]
Then from the equation~\ref{sec3-eq-8} it follows that $\dim|6K_Z| \geq -6+3\cdot 3=3$.

Suppose that $\phi$ has exactly two multiple fibers with multiplicities $m_1$ and $m_2$. Then in order to have $\lambda(\phi)>0$, at least one of them must be greater or equal than $3$. Say $m_1\geq 3$ and $m_2\geq 2$. Then from the equation~\ref{sec3-eq-8} it follows that
\[
\dim|6K_Z| =-6+\left[6(1-\frac{1}{m_1}\right] +\left[6(1-\frac{1}{m_2}\right] \geq -6+\left[ 6\cdot \frac{2}{3}\right]+\left[ 6\cdot \frac{1}{2}\right]=1.
\]
Hence $6K_Z\sim mE$, where $m>0$ is a positive integer and $E$ is a smooth elliptic curve. By repeating now the argument used in Case 2.1 we see that this is impossible if $42K_X^2+3<p$. This concludes the study of the case when $\kappa(Z)=1$.

\end{proof}

\section{The Kodaira dimension  of $Z$ is 0.}\label{sec-5}
Fix the notation as in Section~\ref{sec-2}. The main result of this section is the following.
\begin{proposition}\label{sec5-prop}
Let $X$ be a canonically polarized surface defined over an algebraically closed field of characteristic $p>0$. Suppose that $X$ admits a nontrivial global vector field $D$ such that $D^p=0$ or $D^p=D$. Suppose that $Z$ has Kodaira dimension zero. Then
\[
p<\mathrm{max}\{8(K_X^2)^3+12(K_X^2)^2+3, 4508K_X^2+3\}.
\]
Moreover, suppose that $D^p=D$. Then
\begin{enumerate}
\item Suppose that $K_X^2 =1$. Then $p<179$.
\item Suppose that $K_X^2\geq 2$.  Then $p<140K_X^2+3.$
\end{enumerate} 
\end{proposition}

\begin{proof}
I will only do the case when $K_X^2\geq 2$. The case when $K_X^2=1$ is identical and is omitted. Then only difference between the two cases is that in the Case 3.1 below, where the case when $D^p=D$ is studied, if $K_X^2=1$ then $|4K_X|$ is base point free while if $K_X^2\geq 2$, $|3K_X|$ is base point free~\cite{Ek88}. So in the case $K_X^2=1$,one has to work with the linear system $|4K_X|$ instead.

From now on assume $K_X^2\geq 2$. Suppose that the assumptions of the proposition do not hold, in their respective cases. Then in particular,  $K_X^2<p$. Hence by Theorem~\ref{sec3-lemma-2}, $D$ has only isolated singularities, i.e., $\Delta=0$. Therefore from the equation (\ref{sec2-eq-2}) it follows that  $K_X=\pi^{\ast}K_Y$. 
 Hence, since $K_X$ is ample, $K_Y$ is ample as well. Moreover, $Y$ is singular since if this was not true, then $K_X^2=pK_Y^2\geq p$.
 
Let $f \colon X^{\prime} \rightarrow X$ be the minimal resolution of $X$. Then, as before, since $X$ has canonical singularities, $K_X=f^{\ast}K_Y$ and therefore $X$ is a minimal surface of general type. Moreover, from the equation (\ref{00000}) it follows that \[
12\chi(\mathcal{O}_X)+11K_X^2+1\leq 17K_X^2+37<140K_X^2+3<p.
\]
Hence by Theorem~\ref{sec11-th-2}  every singular point of $X$ is a fixed point of $D$, $D$ lifts to a vector field $D^{\prime}$ on $X^{\prime}$ and that every $f$-exceptional curve is stabilized by $D^{\prime}$.

According to the classification of surfaces~\cite{BM76},~\cite{BM77}, $Z$ is one of the following: An abelian surface, a K3 surface, an Enriques surface or a hyperelliptic  surface.

\textbf{Case 1: Suppose that $Z$ is an abelian surface.} Then  every $g$-exceptional curve is also $h$-exceptional since by Lemma~\ref{sec3-lema-1} every $g$-exceptional curve is rational and there do not exist nontrivial maps from a rational curve to an abelian surface. Hence there exists a factorization
\begin{gather}\label{sec3-eq-9}
\xymatrix{
Y^{\prime}\ar[r]^g\ar[d]_{\phi} & Y \ar[dl]^{\theta}\\
Z & \\
}
\end{gather}
Let $B_j$, $j=1,\ldots, r$ be the $\theta$-exceptional curves. Then one can write 
\[
K_Y=\theta^{\ast}K_Z +\sum_{j=1}^r\gamma_j B_j.
\]
But then, since $\{B_j, \; 1\leq j \leq r\}$ is a contractible set of curves, it easily follows that 
\[
K_Y^2=\left(\sum_{j=1}^r\gamma_j B_j\right)^2\leq 0,
\]
which is impossible since $K_Y$ is ample. Therefore $Z$ cannot be an abelian surface.

\textbf{Case 2: Suppose that $Z$ is a hyperelliptic surface.} I will show that this case is also impossible. It is well known that if $Z$ is hyperelliptic,  then $b_1(Z)=2$~\cite{BM77} and hence $\dim \mathrm{Alb}(Z)=1$. Then the morphism $\phi \colon Z\rightarrow \mathrm{Alb}(Z)$ is an elliptic fibration~\cite{BM77}. Since every $g$-exceptional curve is rational, they must be contracted to points in $\mathrm{Alb}(Z)$. Hence there exists a factorization
\[
\xymatrix{
 & Y \ar[dr]^{\tilde{\psi}} & \\
 X \ar[ur]^{\pi}\ar[rr]^{\psi} & & \mathrm{Alb}(Z)
 }
 \]
 The general fiber $Y_b$ of $\tilde{\psi}$ is an elliptic curve. Hence $K_Y \cdot Y_b=0$. hence
 \[
 K_X\cdot \pi^{\ast}Y_b=\pi^{\ast}K_Y\cdot \pi^{\ast}Y_b=pK_Y\cdot Y_b=0,
 \]
 which is impossible since $K_X$ is ample. Hence $Z$ can be either a K3 surface or an Enriques surface.
 
 \textbf{Case 3: Suppose that $Z$ is a $K3$ surface.} Consider now two cases with respect to whether $D$ is of multiplicative or additive type.
 
 \textbf{Case 3.1.} Suppose that $D$ is of multiplicative type, i.e., $D^p=D$.

 By~\cite[Theorem 1.20]{Ek88}, $|3K_X|$ is base point free. Also, since $K_X=\pi^{\ast}K_Y$,  by Proposition~\ref{sec1-prop1}, there exists a $k$-linear map
 \begin{gather}\label{sec3-eq-10}
  D^{\ast} \colon H^0(\mathcal{O}_X(3K_X)) \rightarrow H^0(\mathcal{O}_X(3K_X)).
\end{gather}
Moreover, since $D^p=D$, $D^{\ast}$ is diagonalizable (with eigenvalues in the set $\{0,1,\ldots, p-1\}$) and their eigenvectors correspond to integral curves of $D$. Let 
\begin{gather}\label{section3-eq-11}
 H^0(\mathcal{O}_X(3K_X))=\oplus_{i=1}^{k} V(\lambda_i),
 \end{gather}
 the decomposition of $H^0(\mathcal{O}_X(3K_X))$ in eigenspaces of $D^{\ast}$, where $\lambda_i \in\mathbb{F}_p$, $1\leq i \leq k$.
 
 Suppose that $\dim |3K_X|=m$. Let  $Z_i$, $i=1,\ldots, m$ be a basis of $|3K_X|$ corresponding to eigenvectors of $D^{\ast}$. Since $K_X$ is ample it follows from~\cite[Corollary 7.9]{Ha77} that $Z_i$ is connected for all $i$. Now since $Z_i$ are eigenvectors of $D^{\ast}$, $Z_i$ are stabilized by $D$ and hence $D$ induces nontrivial vector fields on each $Z_i$. Moreover, if $K_X^2<p/3$, something which is true if the assumptions of the proposition hold,  then from  Corollary~\ref{sec1-cor1}, $D$ restricts to every reduced and irreducible component of $Z_i$, for all $1\leq i \leq m$.
 
 Since $Z$ is a $K3$ surface, $\omega_Z\cong \mathcal{O}_Z$. Hence from the equations (\ref{sec2-eq-1}) it follows that 
 \begin{gather}\label{sec3-eq-12}
 K_Y=\sum_{j=1}^sb_j \tilde{E}_j,
 \end{gather}
 where $\tilde{E}_j$ is the birational transform in $Y$ of the $h$-exceptional curves not contracted by $g$ (note that such curves exist because if this was not the case then $Y^{\prime}=Z$ and hence since $K_Z=0$ it would follow that $K_Y=0$ which is impossible since $K_Y$ is ample). In particular $p_g(Y)\not=0$ and hence 
 $p_g(X)\not=0$. 
 Let $C=\pi^{\ast}\tilde{E}$, where $\tilde{E}$ is any irreducible component of $K_Y$ in the equation (\ref{sec3-eq-12}). Then, since  $K_X^2<p$, $C$ is reduced and hence is an integral curve of $D$. 
 
 \textbf{Claim:} $D$ has at most two fixed points on $C$. 
 
 Indeed. From the equation (\ref{sec3-eq-12}) it follows that
 \begin{gather}\label{sec33-eq-1}
 K_X\cdot C=\pi^{\ast}K_Y \cdot \pi^{\ast} \tilde{E} =pK_Y\cdot \tilde{E}\leq pK_Y^2=K_X^2.
 \end{gather}
 Moreover, from Lemma~\ref{hodge}, $C^2\leq K_X^2$. Let $C^{\prime}=f_{\ast}^{-1}C$ be the birational transform of $C$ in $X^{\prime}$. Then 
 \[
 K_{X^{\prime}}\cdot C^{\prime}=f^{\ast}K_X \cdot C^{\prime}=K_X\cdot C \leq K_X^2.
 \]
 Moreover, $(C^{\prime})^2\leq C^2\leq K_X^2=K_{X^{\prime}}^2$. Therefore $p_a(C^{\prime})\leq K_X^2+1$. Then, the assumptions of the proposition imply that $K_X^2+1<(p-1)/2$. Hence it follows  from Corollary~\ref{sec1-cor2}  that  $D^{\prime}$ fixes the singular point of $C^{\prime}$ and lifts to its normalization $\bar{C}$. Suppose that $C^{\prime}$ is singular. Then by Corollary~\ref{sec1-cor2}, 
 $\bar{C}\cong \mathbb{P}^1$ and $D^{\prime}$ has at most two fixed points on $C^{\prime}$. Suppose that $C^{\prime}$ is smooth. Then it must be either a smooth rational curve or an elliptic curve. In the first case $D^{\prime}$ has exactly two fixed points on $C^{\prime}$. Suppose that $C^{\prime}$ is an elliptic curve. Then the map $C^{\prime} \rightarrow C$ factors through the normalization $\tilde{C}\rightarrow C$. Therefore there exists a purely inseparable map of degree $p$  map $C^{\prime} \rightarrow \tilde{C}$ of smooth curves. Moreover, since $C$ is the pushforward in $Y$ of an $h$-exceptional curve, $C$ is rational and hence $\tilde{C}=\mathbb{P}^1$. Therefore there exists a purely inseparable map of degree $p$, $C^{\prime}\rightarrow \mathbb{P}^1$. But this implies that there exists a map $\mathbb{P}^1 \rightarrow (C^{\prime})^{(p)}$, where $C^{\prime} \rightarrow (C^{\prime})^{(p)}$ is the $k$-linear Frobenius. But this is impossible since $(C^{\prime})^{(p)}$ is also an elliptic curve. Therefore, $C^{\prime}$ cannot be an elliptic curve and hence in any case $D^{\prime}$ has at most two fixed points on $C^{\prime}$.
 
 Next I will show that this implies that $D$ has at most two fixed points on $C$. Let $P \in C$ be a fixed point of $D$. 
 
 Suppose that $P$ is a smooth point of $X$. Then $Q=f^{-1}(P)$ is a fixed point of $D^{\prime}$ on $C^{\prime}$.
 
 Suppose that $P$ is a singular point of $X$. Let then $E$ be an $f$-exceptional curve which intersects $C^{\prime}$. By Theorem~\ref{sec11-th-2}, $E$ is stabilized by $D^{\prime}$. Then by repeating word by word the arguments that lead to the equation (\ref{sec3-eq-111111}) we find that $ E\cdot C^{\prime}\leq 2K_X^2+2$ and hence the assumptions of the proposition imply that $E\cdot C^{\prime} <p$. Therefore, by 
 Corollary~\ref{sec1-cor0}, every point of intersection of $E$ and $C^{\prime}$ is a fixed point of $D^{\prime}$ on $C^{\prime}$. Therefore $D$ has at most as many fixed points on $C$ as $D^{\prime}$ has on $C^{\prime}$ and hence at most 2. This concludes the proof of the claim.

 Therefore $C$ is a rational curve and $D$ has at most two fixed points on $D$. Let $P_1, P_2$  be the fixed points of $D$ on $C$, with the possibility that $P_1=P_2$.
 
Let $1\leq i \leq m$ be such that $C$ is not an irreducible component of $Z_i$.  Since $K_X$ is ample, it follows that $C\cdot Z_i>0$. For the same reason, 
$Z_i \cdot Z_j >0$ and therefore $Z_i \cap Z_j \not=\emptyset, $ for all $1\leq i,j \leq m$. 

Let now again $Z_i$ be a member of the basis of $|3K_X|$.  Let $A$ be an irreducible and reduced  component of $Z_i$ different from $C$ such that $C\cdot A>0$. I will show that every point of intersection of $C$ and $A$ is a fixed point of $D$. Indeed, from the definition of $C$ and $Z_i$, it follows that
 \[
C \cdot A \leq C\cdot Z_i =3K_X \cdot C \leq 3K_X^2<p
\]
by the equation (\ref{sec33-eq-1}) and the assumptions of the proposition. Hence by Corollary~\ref{sec1-cor0},  every point of intersection of $A$ and $C$ which is a smooth point of $X$ is a fixed point of $D$. The points of intersection of $A$ and $C$ which are singular points of $X$ are fixed points of $D$ always. Hence every point of intersection of $C$ and $A$ is a fixed point of $D$. In particular, every point of intersection of $C$ and $Z_i$ is a fixed point of $D$ (in the case $C$ is not a component of $Z_i$).
 
 Suppose that $P_1=P_2$. Let $1\leq i \leq m$. Then either $C$ is a component of $Z_i$ or $(Z_i \cap C)_{\mathrm{red}} =\{P_1\}$, for all $1\leq i\leq m$. But this implies that $P_1$ is a base point of $|3K_X|$, which is impossible. Hence $P_1\not= P_2$. For the same reason, it is not possible that either 
 $(Z_i \cap C)_{\mathrm{red}} =\{P_1\}$, for all $i$ or  $(Z_i \cap C)_{\mathrm{red}} =\{P_2\}$, for all $i$. Therefore there exist indices $1\leq i\not= j \leq m$, such that $(Z_i\cap C)_{\mathrm{red}} =\{P_1\}$ and 
 $(Z_j \cap C)_{\mathrm{red}}=\{P_2\}$. But then, since $Z_i \cap Z_j \not= \emptyset$ and the curves $Z_i$ and $Z_j$ are connected, the curve $W=Z_i +Z_j + C$ contains loops. Let $\tilde{Z}_i=\pi(Z_i)$, 
 $\tilde{Z}_j=\pi(Z_j)$. Then $\tilde{W}=\tilde{Z}_i\ + \tilde{Z}_j\ + \tilde{E}$ is a curve whose reduced curve $\tilde{W}_{\mathrm{red}}$ contains loops. Hence 
 $\dim H^1(\mathcal{O}_{\tilde{W}_{\mathrm{red}}})\geq 1$ and hence  $\dim H^1(\mathcal{O}_{\tilde{W}})\geq 1$ as well.
  
 Now since $Z$ is a $K3$ surface, it follows that $H^1(\mathcal{O}_Z)=0$ Hence $H^1(\mathcal{O}_{Y^{\prime}})=0$ and therefore from the Leray spectral sequence it follows that $H^1(\mathcal{O}_Y)=0$. Then from the exact sequence
 \[
 0 \rightarrow \mathcal{O}_Y(-\tilde{W}) \rightarrow \mathcal{O}_Y \rightarrow \mathcal{O}_{\tilde{W}} \rightarrow 0
 \]
 we get the exact sequence in cohomology
 \[
 \cdots 0=H^1(\mathcal{O}_Y) \rightarrow H^1(\mathcal{O}_{\tilde{W}}) \rightarrow H^2(\mathcal{O}_Y(-\tilde{W}) ) \rightarrow H^2(\mathcal{O}_Y) \rightarrow 
 H^2(\mathcal{O}_{\tilde{W}} )=0.
 \]
 Considering now that $h^1(\mathcal{O}_{\tilde{W}})\geq 1$,  $H^2(\mathcal{O}_Y(-\tilde{W}) )=H^0(\mathcal{O}_Y(\tilde{W}+K_Y))$, $H^2(\mathcal{O}_Y)=H^0(\mathcal{O}_Y(K_Y))$  and that $p_g(Y)\not= 0$, it follows that 
 \begin{gather}\label{sec3-eq-13}
 h^0(\mathcal{O}_Y(\tilde{W}+K_Y))\geq 2.
 \end{gather}
 Now since $\pi^{\ast}(\tilde{W}+K_Y)=W+K_X\sim 7K_X$ it follows that $|7K_X|$ contains a positive dimensional subsystem whose members are stabilized by $D$. Then by Proposition~\ref{sec1-prop-8}, and the assumptions of the proposition, this is impossible. Hence $Z$ cannot be a K3 surface.

 \textbf{Case 3.2.} Suppose that $D$ is of additive type, i.e., $D^p=0$.

 The main idea in order to treat this case this is the following. I will show that there exists a "small" positive number $\nu$ such that $\dim|\nu K_Y| \geq 1$ and then get a contradiction for $p$ large enough by Corollary~\ref{sec1-cor-8}.
 
 The main steps of the proof are the following.
 
 Let $F=\sum_{j=1}^nF_j$ be the reduced $g$-exceptional divisor. Then write $F=F^{\prime}+F^{\prime\prime}$, where $F^{\prime}=\sum_{j=1}^rF_j $, where $F_{j}$, $j=1,\ldots,r$ are the $g$-exceptional curves which are not $h$-exceptional, and $F^{\prime\prime}=\sum_{j=r+1}^nF_j$ are the $g$-exceptional curves that are also $h$-exceptional. Notice that $F^{\prime}\not=0$ because if that was the case then there would be a birational morphism $\psi \colon Y \rightarrow Z$. Then by the adjunction formula, 
 \[
 K_Y=\psi^{\ast}K_Z+ \tilde{F}=\tilde{F},
 \]
since $K_Z=0$, where $\tilde{F}$ is a $\psi$-exceptional divisor. Then $K_Y^2=\tilde{F}^2\leq 0$, which is impossible since $K_Y$ is ample.
 
 Then I will show that at least one of the following is true.
 
  \begin{enumerate}
 \item $h^0(\mathcal{O}_Y(2K_Y))\geq 2$ and hence $\dim|2K_Y|\geq 1$.
 \item There exists a divisor $B=\sum_{j=1}^r n_j F_j$, and a positive number $\nu$ such that either $\nu\leq K_X^2$ or $\nu=44$, and such that $\dim |\nu K_{Y^{\prime}}+B| \geq 1$. Moreover, the linear system $|\hat{B}|$, where $\hat{B}=h_{\ast}B$ in $Z$ is either base point free or its moving part is base point free. This implies that $\dim|\nu K_Y| \geq 1$. 
 \end{enumerate}
Then in both cases the claimed result will be obtained by using Corollary~\ref{sec1-cor-8}.

 The assumptions of the proposition imply that $10K_X^2+3<p$ and $4K_X^2+3<p$. Therefore, by Corollary~\ref{sec1-cor-8},  
 \begin{gather}\label{sec33-eq-4}
 H^0(\mathcal{O}_Y(K_Y))=H^0(\mathcal{O}_Y(2K_Y))=k
 \end{gather}
 and hence $p_g(Y)=1$.

 Next I will show that $Y$ has rational singularities. Indeed. The Leray spectral sequence for $g$ gives 
 \[
 0 \rightarrow H^1(\mathcal{O}_{Y}) \rightarrow H^1(\mathcal{O}_{Y^{\prime}})\rightarrow H^0(R^1g_{\ast}\mathcal{O}_{Y^{\prime}}) 
 \rightarrow H^2(\mathcal{O}_Y)\rightarrow H^2(\mathcal{O}_{Y^{\prime}}) \rightarrow H^1(R^1g_{\ast}\mathcal{O}_{Y^{\prime}}).
 \] 
 Now since $g$ is birational it follows that $H^1(R^1g_{\ast}\mathcal{O}_{Y^{\prime}})=0$. Moreover, by Serre duality, $H^2(\mathcal{O}_Y)\cong H^0(\mathcal{O}_Y(K_Y))=k$ and 
 $H^2(\mathcal{O}_{Y^{\prime}})\cong H^0(\mathcal{O}_{Y^{\prime}}(K_{Y^{\prime}}))=k$ and $H^1(\mathcal{O}_{Y^{\prime}})=0$, since $Z$ is a $K3$ surface. Hence from the Leray sequence it follows that $H^1(\mathcal{O}_Y)=0$ and  $ R^1g_{\ast}\mathcal{O}_{Y^{\prime}}=0$. Therefore $Y$ has rational singularities as claimed. In particular, every $g$-exceptional curve is a smooth rational curve.

 Let $\hat{F}_i=h_{\ast}F_i$, $i=1,\dots, r$, be the birational transforms of the $F_i$ in $Z$. Consider next cases with respect to whether the curves $\hat{F}_i$ are either all smooth or there exists a singular one among them. 
 
 
 \textit{Case 1.} Suppose that there exists an $1\leq i \leq r$ such that $\hat{F}_i$ is singular. In this case I will show that $\dim |(K_X^2)K_Y|\geq 2$ and then get a contradiction by Corollary~\ref{sec1-cor-8}.

After a renumbering of the $g$-exceptional curves we can assume that $i=1$. Then by the adjunction formula
 \[
 \hat{F}_1^2=2p_a(\hat{F}_1)-2-K_Z\cdot \hat{F}_1 =2p_a(\hat{F}_1)-2\geq 0.
 \]
 Hence the linear system $|\hat{F}_1|$ in $Z$ is base point free~\cite[Propositions 3.5,  3.10]{Hu16}.
 
\begin{claim}\label{sec3-claim-1}  Let $Q\in \hat{F}_1$ be a singular point of $\hat{F}_1$ and $m=m_Q(\hat{F}_1)$ be the multiplicity of the singularity. Then
\begin{gather}\label{sec3-eq-14}
m_Q(\hat{F}_1) \leq K_X^2.
\end{gather} 
 \end{claim}
In order to prove the claim, observe the following. Over a neighborhood of any singular point of $\hat{F}_1$, $F_1$ can meet at most two distinct $h$-exceptional curves $E_i$ and $E_j$, and moreover it must intersect each one of them with multiplicity 1. Indeed.

Suppose that $F_1$ meets three distinct $h$-exceptional curves $E_{i}$, $E_j$ and $E_s$ (over the same point of $Z$). Since $h$ is a composition of blow ups, it follows that  $E_i\cap E_j \cap E_s=\emptyset$. Hence the intersection of $F_1$ and $E_i\cup E_j \cup E_s$ consists of at least two distinct points, say $P$ and $Q$. Up to a change of indices we can assume hat $P\in E_i$ and $Q\in E_j$.  Then the union $\mathrm{Ex}(h)\cup F_1$, where $\mathrm{Ex}(h)$ is the exceptional set of $h$, contains a cycle. Therefore from the equations (\ref{sec2-eq-1}) it follows that
\begin{gather}\label{sec3-eq-14}
K_Y=\sum_{j=1}^n b_j\tilde{E}_j,
\end{gather}
where $\tilde{E}_j=g_{\ast}E_j$, $j=1,\ldots, n$. Moreover if $E_j^2=-1$, then $\tilde{E}_j\not=0$. But then, if $F_1$ meets at least two distinct $h$-exceptional curves, $\cup_{j=1}^n \tilde{E}$ contains either a singular curve or a cycle. In any case,  if $\tilde{C}=\sum_{j=1}^n b_j\tilde{E}_j$ then  
$H^1(\mathcal{O}_{\tilde{C}})\not=0$. But then from the equation in cohomology
\[
H^1(\mathcal{O}_Y) \rightarrow H^1(\mathcal{O}_{\tilde{C}}) \rightarrow H^2(\mathcal{O}_Y(-\tilde{C}))\rightarrow H^2(\mathcal{O}_Y)\rightarrow 0,
\]

and since $H^1(\mathcal{O}_Y)=0$, $H^2(\mathcal{O}_Y)=k$, it follows that $\dim H^2(\mathcal{O}_Y(-\tilde{C}))\geq 2$.  Then by duality,
\[
h^0(\mathcal{O}_Y(K_Y+\tilde{C}))=h^0 (\mathcal{O}_Y(2K_Y))\geq 2,
\]
a contradiction to the equations (\ref{sec33-eq-4}). Hence $F_1$ meets at most two distinct $h$-exceptional curves. Suppose that $F_1$ meets an $h$-exceptional curve $E_i$ and $E_i \cdot F_1 \geq 2$. Then there are two possibilities. Either $E_i$ is also $g$-exceptional or it is not. Suppose that $E_i$ is $g$-exceptional. But this is impossible because $Y$ has rational singularities and in such a case two $g$-exceptional curves cannot intersect with multiplicity bigger than one. Suppose that $E_i$ is not $g$-exceptional. Then $\tilde{E}_i=g_{\ast}E_i$ is singular and therefore $h^1(\mathcal{O}_{\tilde{E}_i})\geq 1$. But then $h^1(\mathcal{O}_{\tilde{C}})\geq 1$ and hence arguing as before we see that $h^0(\mathcal{O}_Y(2K_Y))\geq 2$, which is again a contradiction to the equations (\ref{sec33-eq-4}). Hence it has been shown that over a neighborhood of any singular point of $\tilde{F}_1$, $F_1$ meets at most two $h$-exceptional curves with multiplicity at most one. 

Next I will show that 
\begin{gather}\label{sec3-eq-15}
m_Q(\hat{F_1}) \leq K_{Y^{\prime}}\cdot F_1.
\end{gather}
The map $h$ is a composition of blow ups of points of $Z$. Since $\hat{F}_1$ is singular, $h$ must blow up the singular points of $\hat{F}_1$. Let $h_1 \colon Y_1 \rightarrow Z$ be the blow up of $Q\in Z$. Then there exists a factorization
\[
\xymatrix{
Y^{\prime}\ar[dr]^{h_2}\ar[rr]^h & & Z \\
    & Y_1\ar[ur]^{h_1} & 
    }
\]
Then also $h_1^{\ast}\hat{F}_1= (h_1)_{\ast}^{-1}\hat{F}_1 +m_Q(\hat{F}_1)E_1$, where $E_1$ is the $h_1$-exceptional curve and $  (h_1)_{\ast}^{-1}\hat{F}_1$ is the birational transform of $\hat{F}_1$ in $Y_1$.     
From this it follows that $E_1\cdot (h_1)_{\ast}^{-1}\hat{F}_1=m_Q(\hat{F}_1)$.  Also $K_{Y_1}=h_1^{\ast}K_Z+E_1=E_1$. Therefore $K_{Y_1}\cdot (h_1)_{\ast}^{-1}\hat{F}_1=m_Q(\hat{F}_1)$. Moreover,  
\[
K_{Y^{\prime}}=h_2^{\ast}K_{Y_1}+E^{\prime},
\]
where $E^{\prime}$ is an effective $h_2$-exceptional divisor. But then
\[
K_{Y^{\prime}}\cdot F_1 = h_2^{\ast}K_{Y_1}\cdot F_1+E^{\prime}\cdot F_1 \geq K_{Y_1}\cdot  (h_2)_{\ast}F_1 =K_{Y_1}\cdot (h_1)_{\ast}^{-1}\hat{F}_1 =m_Q(\hat{F}_1).
\]
This proves the claim. 

As it has been shown earlier, $F_1$ meets at most two $h$-exceptional curves $E_j$ and $E_s$, with the possibility $j=s$, each one of them  with intersection multiplicity one. 

Suppose that $E_i \not= E_j$ and that $F_1$ intersects $E_j$ and $E_s$ at the same point $Q$. Hence $E_j\cap E_s \cap F_1 \not=\emptyset$. Then since $Y$ has rational singularities it is not possible that $E_j$ and $E_s$ are both $g$-exceptional. 

Suppose that $E_s$ is $g$-exceptional but $E_j$ is not $g$-exceptional. Then $g_{\ast}E_j$ would be singular.  But then from the equation (\ref{sec3-eq-14}) and the arguments following it, we get again 
 that $\dim H^0(\mathcal{O}_Y(2K_Y)) \geq 2$, a contradiction to the equation (\ref{sec33-eq-4}).
 
 Hence neither of $E_j$ and $E_s$ is $g$-exceptional. Now write
 \[
 K_{Y^{\prime}}=b_jE_j+b_sE_s+\sum_{r\not= j,s}b_r E_r.
 \]
 Then from the equation (\ref{sec3-eq-15}) and the facts that $E_j\cdot F_1=E_s\cdot F_1=1$, $F_1 \cdot E_r=0$, for $r\not=j,s$,  it follows that 
 \[
 m_Q(\hat{F}_1)\leq K_{Y^{\prime}}\cdot F_1 =b_j+b_s.
 \]
 Then from the equation (\ref{sec3-eq-14}) and the fact that $E_j$ and $E_s$ are not $g$-exceptional it follows that 
 \[
 K_Y=b_j\tilde{E}_j+b_s\tilde{E}_s+\tilde{W},
 \]
 where $W$ is an effective divisor. Then since $K_X=\pi^{\ast}K_Y$ we get that 
\[
K_X=b_j\pi^{\ast}\tilde{E}_j+b_s\pi^{\ast}\tilde{E}_s+\pi^{\ast}\tilde{W}.
\]
Now considering that $K_X$ is ample we get that 
\[
m_Q(\hat{F}_i)\leq K_{Y^{\prime}}\cdot F_i =b_j+b_s\leq b_j\pi^{\ast}\tilde{E}_j \cdot K_X+b_s \pi^{\ast}\tilde{E}_s \cdot K_X\leq K_X^2,
\]
as claimed.

Suppose finally that $E_j=E_s$, i.e., $F_1$ meets exactly one $h$-exceptional curve. Then $K_{Y^{\prime}}\cdot F_1=b_j$. If $E_j$ is not $g$-exceptional then the previous argument proves the claim. Suppose that $E_j$ is also $g$-exceptional. Then there exists a $-1$ $h$-exceptional curve $E_{\lambda}$ such that $b_{\lambda}\geq b_j$. The previous argument now shows that $b_{\lambda} \leq K_X^2$ and hence \[
m_Q(\hat{F}_i) \leq b_j\leq b_{\lambda}\leq K_X^2.
\]
This concludes the proof of Claim~\ref{sec3-claim-1}.

\begin{claim}\label{sec3-claim-2} Let $B$ be any member of the linear system $|(K_X^2)K_{Y^{\prime}}+F_1|$.  Then  
\begin{gather}\label{sec3-eq-16}
B\sim W^{\prime}+\sum_{i=1}^m \gamma_i E_i,
\end{gather}
where $\gamma_i\geq 0$ for all $i$ and $W^{\prime}$ is the birational transform in $Y^{\prime}$ of a smooth curve $W$ in $Z$ such that  $|W|$ is base point free and $\mathrm{p}_a(W) \geq 1$.
\end{claim}

By~\cite[Proposition 3.5 and 3.10]{Hu16}, the linear system $|\hat{F}_1|$ is base point free and contains a smooth curve. Let $W \in |\hat{F}_1|$. be a general member. Then $W$ is reduced and irreducible and moreover it does not pass through $h(\mathrm{Ex}(h))$. Let $W^{\prime}$ be the birational transform of $W$ in $Y$. Then $W^{\prime}\cong W$. Now from Proposition~\ref{sec1-prop-6} it follows that 
\begin{gather}\label{sec3-eq-17}
\mu K_{Y^{\prime}}-h^{\ast}\hat{F}_1+F_1=\sum_{i=1}^m\gamma_i E_i,
\end{gather}
where $\gamma_i \geq 0$, for all $1\leq i \leq m$, and $\mu$ is the maximum of the multiplicities of the singular points of $\hat{F}_1$. But from Claim~\ref{sec3-claim-1} it follows that $\mu \leq K_X^2$. Hence 
\begin{gather}\label{sec3-eq-18}
(K_X^2)K_{Y^{\prime}}-h^{\ast}\hat{F}_1+F_1=\sum_{i=1}^m\gamma^{\prime}_i E_i,
\end{gather}
for some $\gamma_i^{\prime}\geq 0$, for all $1\leq i \leq m$. Let now  $W \in |\hat{F}_1|$ be a general member. Then $W^{\prime}=h^{\ast}\hat{F}_1=F_i +(h^{\ast}\hat{F}_1-F_1)$. Then from the equation (\ref{sec3-eq-17}) it follows that
\[
(K_X)^2K_{Y^{\prime}}+F_i=(K_X^2)K_{Y^{\prime}}+W^{\prime}-h^{\ast}\hat{F}_i+F_i=W^{\prime}+\sum_{i=1}^m\gamma^{\prime}_i E_i,
\]
for some $\gamma_i^{\prime}\geq 0$, $1\leq i \leq m$. This concludes the proof of Claim~\ref{sec3-claim-2}.

Now pushing down to $Y$ by $g_{\ast}$, and considering that $F_i$ is $g$-exceptional, we see that 
\begin{gather}\label{sec3-eq-19}
(K_X^2)K_Y\sim \tilde{W}+\sum_{j=1}^m\gamma_j^{\prime}\tilde{E}_j.
\end{gather}
 Moreover notice that from the construction of $W$, $\dim|\tilde{W}| \geq 1$ and therefore $\dim |(K_X^2)K_Y|\geq 1$. But according to the assumptions of the proposition, 
 \[
 (K_X^2)^3+3(K_X^2)^2+3<p,
 \]
 and therefore from Corollary~\ref{sec1-cor-8} we get a contradiction. Hence there is no $g$-exceptional curve $F_i$ such that $\hat{F}_i=h_{\ast}F_i$ is singular.
 
 \textit{Case 2.} Suppose that $\hat{F}_i$ is smooth for any $i=1,\ldots, r$. In this case I will show that $ K_X^2> (p-3)/506$.

 Since $\hat{F}_i$ is smooth it follows that $\hat{F}_i\cong \mathbb{P}^1$ and that $\hat{F}_i^2=-2$, for all $i=1,\dots,r$. Consider now cases with respect to whether or not every connected subset of the 
 set $\{\hat{F},\ldots,\hat{F}_r\}$ is contractible. 
 
 \textit{Case 2.1.} Suppose that every connected subset of $\{\hat{F}_1,\ldots,\hat{F}_r\}$ is contractible. Let $\phi \colon Z \rightarrow W$ be the contraction. Since $\hat{F}_i^2=-2$, for all $i=1,\ldots, r$, $W$ has Duval singularities. Therefore $K_Z=\phi^{\ast}K_W$. Hence,  since $K_Z=0$, $K_W=0$. Then there exists a factorization
 \[
 \xymatrix{
 Y^{\prime}\ar[rr]^g\ar[dr]^{\phi h} & & Y\ar[ld]^{\psi}\\
        & W &
 }
 \]
 Hence $K_Y=\psi^{\ast}K_W+\tilde{E}=\tilde{E}$, where $\tilde{E}$ is a divisor supported on the $\psi$-exceptional set. But then $K_Y^2=\tilde{E}^2\leq 0$, which is impossible since $K_Y$ is ample.
 
 \textit{Case 2.2.} There exists at least one connected subset of $\{\hat{F}_1,\ldots,\hat{F}_r\}$ which is not contractible. 
 
 \begin{claim} There exists integers $0 \leq \gamma_j \leq 22$, $j=1,\ldots, r$ such that the linear system $|44K_{Y^{\prime}}+\sum_{j=1}^r\gamma_j F_j|$ has dimension at least one. Moreover, let 
 $B\in |44K_{Y^{\prime}}+\sum_{j=1}^r\gamma_j F_j| $ be any member. Then if $K_X^2 <p/44$, 
 \[
 B\sim W^{\prime}+\sum_{i=1}^m \gamma_i E_i,
\]
where $\gamma_i\geq 0$ for all $i$ and $W^{\prime}$ is the birational transform in $Y^{\prime}$ of a reduced and irreducible curve $W$ in $Z$ such that  $|W|$ is base point free and $\mathrm{p}_a(W) \geq 1$. 

 \end{claim}\label{sec3-claim-5}
 
 In order to prove the Claim~\ref{sec3-claim-5} it is necessary to prove first the following.
 
 \begin{claim}\label{sec3-claim-6} 
There exist numbers $0 \leq \gamma_i \leq 23$, $i=1,\dots, r$ such that if $\Gamma=\sum_{i=1}^r\gamma_i \hat{F}_i$, then $\Gamma \cdot \hat{F}_i \geq 0$, for all $1\leq i \leq r$, and $\Gamma^2\geq 0$.
\end{claim}

I  proceed to prove the claims. Let $\{\hat{F}_1,\ldots, \hat{F}_s\}$, $s<r$, be the maximal connected subset of $\{\hat{F}_1,\ldots,\hat{F}_r\}$ which is contractible. Since the rank of $\mathrm{Pic}(Z)$ is at 
most 22~\cite{Hu16} it follows that $s\leq 22$. 


 Let $\phi \colon Z \rightarrow Z^{\prime}$ be the contraction of $\{\hat{F}_1,\ldots, \hat{F}_s\}$. Then $Z^{\prime}$ has Du Val singularities. Since $\cup_{i=1}^r\hat{F}$ is connected, there exists a curve $\hat{F}_j \in \{\hat{F}_{s+1},\ldots, \hat{F}_r\}$,  such that $\hat{F}_j\cap (\cup\hat{F}_{i=1}^s)\not=\emptyset$ and of course $\hat{F}_j$ does not contract by 
$\phi$. Let $F_j^{\prime}=\phi_{\ast}\hat{F}_j$. Observe now that one of the following happens.
\begin{enumerate}
\item $F^{\prime}_j$ is singular. In this case one of the following happens.
\begin{enumerate}
\item $\hat{F}_j$  meets two distinct $\phi$-exceptional curves, say $\hat{F}_{\lambda}$, $\hat{F}_{\mu}$, $1\leq \lambda < \mu \leq s$.  
\item $\hat{F}_j$ meets one $\phi$-exceptional curve $\hat{F}_i$, $i \leq s$, such that $\hat{F}_j \cdot \hat{F}_i \geq 2$. 
\item $\hat{F}_j$ meets exactly one $\phi$-exceptional curve $\hat{F}_i$ and $\hat{F}_i \cdot \hat{F}_j =1$.
\end{enumerate}
\item $F^{\prime}_j$ is smooth.
\end{enumerate}
 
 Suppose that the case 1.a happens. Then let $\Gamma=\hat{F}_j +\sum_{i=\lambda}^{\mu} \hat{F}_i$. Then this is a cycle of $-2$ rational curves and $\Gamma \cdot \hat{F}_i =0$, for all $i\in \{j, \lambda, \lambda +1,\ldots, \mu\}$, and $\Gamma^2=0$.
 
 Suppose that the case 1.b happens. Then let $\Gamma=\hat{F}_j+\hat{F}_i$. Then $\Gamma \cdot \hat{F}_j \geq 0$, $\Gamma \cdot \hat{F}_i \geq 0$ and $\Gamma^2\geq 0$.
 
 Suppose that the case 1.c happens. This can happen only when the fundamental cycle of the singularity of $W$ is not reduced, i.e., when $W$ has either a $D_s$, $E_6$, $E_7$ or $E_8$ singularity. 
 
 Suppose that $W$ has a $D_s$ singularity.  The fundamental cycle of the singularity is $\hat{F}_1 +2\sum_{i=1}^{s-2}\hat{F}_i +\hat{F}_{s-1}+\hat{F}_s$. Hence in this case  $\hat{F}_j$ must intersect some $\hat{F}_i$, $2\leq i \leq s-2$. Let $\Gamma=\hat{F}_j+\hat{F}_{i-1}+2\sum_{k=1}^{s-2}\hat{F}_k +\hat{F}_{s-1}+\hat{F}_s$. Then  $\Gamma \cdot \hat{F}_j = 0$, $\Gamma \cdot \hat{F}_k = 0$, $i-1\leq k \leq s$ and $\Gamma^2=0$. 
 
 The cases when $W$ has $E_6$, $E_7$ or $E_8$ singularities are treated similarly. 
 
 Suppose finally that case 2 happens, i.e., $F^{\prime}_j$ is smooth. Then write
 \[
 \phi^{\ast}F_j^{\prime}=\hat{F}_j+\sum_{i=1}^sa_i\hat{F}_i.
 \]
 Let $m$ be the index of $F^{\prime}_j$ in $S$. Then according to Proposition~\ref{sec1-prop-7}, $m\in\{2,3,4, s+1\}$ (the exact value of $m$ depends on the type of singularities of $S$). Moreover, if $S$ has an $A_s$ or $D_s$ singularity, then $ma_i \leq s$, for all $i=1,\ldots, s$. If $S$ has an $E_6$ singularity then $ma_i \leq 6$ for all $i$ and if $S$ has an $E_7$ singularity then $ma_i \leq 7$ for all $i$. In any case $ma_i$ are positive integers at most 22, for all $i=1,\ldots, s$, and $m\leq s+1\leq 23$. Let $\gamma_i= ma_i$, for all $i=1,\ldots, s$ and $\gamma_j=m$. Let also
 \[
 \Gamma=m\phi^{\ast} F_j^{\prime}=\gamma_j\hat{F}_j+\sum_{i=1}^s\gamma_i\hat{F}_i.
 \]
 Then $\Gamma \cdot \hat{F}_i =0$, $i=1,\dots, s$, and $\Gamma \cdot \hat{F}_j=m(F^{\prime}_j)^2\geq 0$ (if $(F^{\prime})^2<0$, then the set $\{\hat{F}_j, \hat{F}_1, \ldots, \hat{F}_s\}$ would be contractible which is not true). Moreover, $\Gamma^2\geq 0$. This concludes the proof of Claim~\ref{sec3-claim-6}.
 
 So it has been proved that there exists a nontrivial effective divisor $\Gamma=\sum_{i=1}^r \gamma_i \hat{F}_i$ in $Z$, such that $0\leq \gamma_i \leq 23$, $i=1,\dots,r$, and $\Gamma \cdot \hat{F}_i \geq 0$ for all $i=1,\dots, r$ and $\Gamma^2\geq 0$. In particular,  if three of the $\hat{F}_i$ meet at a common point or two have a tangency then $B$ is reduced.  Now since $\hat{F}_i$ is smooth for all $i$, every multiple $\gamma_i \hat{F}_i$ can be considered singular with multiplicity $\gamma_i \leq 23$ at every point. If two, say $\hat{F}_i$ and $\hat{F}_j$ meet at a point with multiplicity 1 then $\Gamma$ has at this point multiplicity $\gamma_i+\gamma_j\leq 23+23=46$. Therefore from Proposition~\ref{sec1-prop-6} it follows that $46K_{Y^{\prime}}-h^{\ast}\Gamma+\Gamma^{\prime}$ is an effective divisor, where $\Gamma^{\prime}=\sum_{i=1}^r \gamma_i F_i$. 
 
 Consider now cases with respect to $\Gamma^2$.
 
 Suppose that $\Gamma^2=0$. Then by~\cite[Proposition 3.10]{Hu16}, the linear system $|B\Gamma$ is base point free. Moreover, by~\cite[Theorem 6.3]{Jou83}, if $p\not=2,3$, 
 $\Gamma\sim p^{\nu}W$, where $W$ is a smooth irreducible elliptic curve. In fact $|\Gamma|$ is also base point free~\cite[Proposition 3.10]{Hu16}. I claim that if $\nu >0$, then $K_X^2>p^{\nu}/44$. Indeed. 
 \begin{gather}\label{sec3-eq-100}
 46K_{Y^{\prime}}+\Gamma^{\prime}=46K_{Y^{\prime}}+\Gamma^{\prime}-h^{\ast}\Gamma+h^{\ast}\Gamma=(46K_{Y^{\prime}}-h^{\ast}\Gamma+\Gamma^{\prime})+p^{\nu}W^{\prime}=\\
 p^{\nu}W^{\prime} +E,\nonumber
 \end{gather}
 where $E$ is an effective divisor whose prime components are $g$-exceptional and $h$-exceptional curves and $W^{\prime}$ is the birational transform of $W$ in $Y^{\prime}$ ($W$ can be chosen to avoid the points blown up by $h$). Then by pushing down to $Y$ and then pulling up on $X$ we find that
 \begin{gather}\label{sec3-eq-22}
 46K_X=p^{\nu}\pi^{\ast}\tilde{W}+\pi^{\ast}\tilde{E},
 \end{gather}
 where $\tilde{W}=g_{\ast}W$ and $\tilde{E}=g_{\ast}E$. Also notice that since $W$ moves in $Z$, $W^{\prime}$ is not $g$-exceptional and hence $\tilde{W}\not=0$. Then, since $K_X$ is ample, it follows that $46K_X^2\geq p^{\nu}$. But this is impossible since we are assuming that the inequalities of the statement of Claim 3.2 do not hold. Hence  $\nu=0$. Then by pushing the equation~\ref{sec3-eq-100} down to $Y$ we get that
 \[
 46K_Y=\tilde{W}+g_{\ast}E,
 \]
 where $\tilde{W}$ is the birational transform of $W$ in $Y$. Now since $\dim |W|\geq 1$ it follows that $\dim |46K_Y|\geq 1$. But according to the assumptions of the proposition,
 \[
 (46\cdot 49)K_X^2+3<p.
 \]
 Then from Corollary~\ref{sec1-cor-8} gives a contradiction. So it is not possible that $\Gamma^2=0$.

 Suppose finally that $\Gamma^2>0$. Then $\Gamma$ is nef and big. Then by~\cite[Corollary 3.15]{Hu16}, $\Gamma \sim mW+C$, where $W$ is a smooth elliptic curve and $C\cong \mathbb{P}^1$. Moreover, as before, the linear system $|W|$ is base point free~\cite[Proposition 3.10]{Hu16}. Repeating now the arguments of the previous case we find that
 \[
 46K_{Y^{\prime}}+\Gamma^{\prime}=mW^{\prime}+C^{\prime}+E,
 \]
 where $W^{\prime}$ and $C^{\prime}$ are the birational transforms of $W$ and $C$ in $Y^{\prime}$ and $E$ is effective. Repeating now word by word the arguments of the case when $\Gamma^2=0$ we get again  a contradiction.

Therefore, under the assumptions of the proposition, $Z$ is not a $K3$ surface.


 \textbf{Case 4: Suppose that $Z$ is Enriques.} 
  
 In this case, since we assume $p\not= 2$, $\pi_1(X)=\pi_1(Z)=\mathbb{Z}/2\mathbb{Z}$. Then there exists an \'etale double cover  
 $\nu \colon W\rightarrow X$ of $X$. Then $K_W=\nu^{\ast}K_X$ and $K_W^2=2K_X^2$. Also $D$ lifts to a nontrivial global vector field $D^{\prime}$ on $W$. Then in the corresponding diagram (\ref{sec2-diagram-1}) for $W$, $Z$ is going to be a $K3$ surface. Then, under the assumptions of the proposition,  the results from the previous cases for $W$ show that $Z$ cannot also be an Enriques surface.

 \end{proof}

\section{The Kodaira dimension of $Z$ is $-1$.}\label{sec-6}
\begin{proposition}\label{sec6-prop}
Let $X$ be a  canonically polarized surface defined over an algebraically closed field of characteristic $p>0$.  Suppose that $X$  admits a nontrivial global vector field $D$ such that $D^p=0$ or $D^p=D$.
Suppose  also that, with notation as in section~\ref{sec-3},   $Z$ has Kodaira dimension $-1$ and that one of the following holds
\begin{enumerate}
\item $K_X^2=1$ and $p>211$.
\item $K_X^2\geq 2$ and $p> 156K_X^2+3$.
\end{enumerate}
Then $X$ is unirational and $\pi_1(X)=\{1\}$.
\end{proposition}
\begin{proof}
 
 I will only do the case when $K_X^2\geq 2$. The only difference between the two case is that in the proof  one has to use the inequalities in Proposition~\ref{sec11-prop-5} that correspond to each case. Otherwise the proofs are identical.
 
 Let $f \colon X^{\prime}\rightarrow X$ be the minimal resolution of $X$. Then from the inequalities (\ref{00000}) it follows that
 \[
 12\chi(\mathcal{O}_X)+11K_X^2+1\leq 17K_X^2+37<156K_X^2+3<p,
 \]
 from the assumptions. Therefore from Theorem~\ref{sec11-th-2}, $D$ lifts to a vector field $D^{\prime}$ on $X^{\prime}$. Moreover, every $f$-exceptional curve is stabilized by $D^{\prime}$.

 Since $\kappa(Z)=-1$, $Z$ is a ruled surface. Hence there exists a fibration of smooth rational curves $\phi \colon Z \rightarrow B$, where $B$ is a smooth curve. 
 
 \textbf{Claim:} Under the conditions of the proposition, $B\cong \mathbb{P}^1$.

 Suppose that the claim has been proved. Then $Z$ and hence $Y^{\prime}$ are rational. In particular $\pi_1(Y^{\prime})=\pi_1(Z)=\{1\}$. Then there exists a commutative diagram
 \begin{gather*}
 \xymatrix{
 Y^{\prime}_{(p)} \ar[r]^{\sigma} \ar[dr]_{F_{(p)}} & \hat{X} \ar[r]^{\hat{g}} \ar[d]^{\hat{\pi}} & X \ar[d]^{\pi}\\
   & Y^{\prime}\ar[r]^g  &Y
 }
 \end{gather*}
 Where $\hat{X}$ is the normalization of $Y^{\prime}$ in $X$, $\hat{\pi}$ and $\sigma$ are purely inseparable maps of degree $p$, $\hat{g}$ is birational and $F_{(p)}$ is the $k$-linear Frobenius. Therefore, since $Y^{\prime}$ is rational, $Y^{\prime}_{(p)}$ is also rational and hence $X$ is purely inseparably unirational. Moreover, since $\hat{\pi}$ is purely inseparable, it follows that $\pi_1(\hat{X})=\pi_1(Y^{\prime})=\{1\}$. Then, since $\hat{g}$ is birational and $\hat{X}$ and $X$ are normal, it follows by~\cite[Chapter X]{Gr60} that the natural map $\pi_1(\hat{X}) \rightarrow \pi_1(X)$ is surjective. Therefore $\pi_1(X)=\{1\}$.
 
 Therefore it remains to prove the claim. 
 
  Suppose that a $g$-exceptional curve $F$ does not map to a point in $B$ by the map $\phi h$. Then there exists a dominant morphism $F\rightarrow B$. But since $F$ is a rational curve then $B\cong \mathbb{P}^1$. 
 
 Suppose that every $g$-exceptional curve is contracted to a point in $B$ by $\phi h$. Then there exists a factorization
 \begin{gather}\label{sec66-diagram-1}
 \xymatrix{
 Y^{\prime}\ar[r]^g\ar[d]^h & Y \ar[d]^{\psi}\\
 Z \ar[r]^{\phi} & B
 }
 \end{gather}
The general fiber of $\psi$ is a smooth rational curve. Also, since the $g$-exceptional set is contained in fibers of $\phi h$, $Y$ has rational singularities. Let $\sigma \colon X \rightarrow B$ be the composition $\psi\pi$.

 Consider next cases with respect to whether the divisorial part $\Delta$ of $D$ is zero or not.
 
 \textbf{Case 1:  $\mathbf{\Delta=0}$.}
 
 Then $K_X=\pi^{\ast}K_Y$ and hence, since $\pi$ is a finite map,  $K_Y$ is ample.  Let $Y_b$ be a general fiber of $\psi$. Then $Y_b\cong \mathbb{P}^1$. Therefore since $Y_b^2=0$, it follows that $K_Y \cdot Y_b=-2$,which is impossible since $K_Y$ is ample. Therefore in this case $B\cong \mathbb{P}^1$.

 \textbf{Case 2: $\mathbf{\Delta\not= 0}$.}

In order to show that $B\cong \mathbb{P}^1$ I will show that there exists a rational curve (in general singular) $C$ in $X$ which dominates $B$. The method to find such a rational curve is to show that there exists an integral curve $C$ of $D$ on $X$ which dominates $B$. Then by Corollary~\ref{sec1-cor2}, if the arithmetic genus of $C$ is small compared to the characteristic $p$, $C$ is rational. Finally, integral curves of $D$ will be found by utilizing Proposition~\ref{sec1-prop1}.

By~\cite[Theorem 1.20]{Ek88}, the linear system $|3K_X|$ is base point free. Then by~\cite[Theorem 6.3]{Jou83}, the general member of $|3K_X|$ is of the form $p^{\nu} C$, where $C$ is an irreducible and reduced curve. Suppose that $\nu >0$. Then $K_X^2>p/3$, which is impossible from the assumptions of the proposition. Hence the general member of $|3K_X|$ is reduced and irreducible (but perhaps singular). 

The assumptions of the proposition imply that $p$ does not divide $K_X^2$. Therefore, from Proposition~\ref{sec11-prop-5} it follows that
\begin{gather}\label{sec4-eq-1000}
K_X\cdot \Delta \leq 3K_X^2 \\\nonumber
\Delta^2\leq 9K_X^2.
\end{gather}

\textbf{Claim:} There exists a rank 1 reflexive sheaf $M$ on $Y$ such that 
\[
\mathcal{O}_X(K_X+\Delta)=(\pi^{\ast} M)^{[1]}.
\]
I proceed to prove the claim. According to the adjunction formula (\ref{sec2-eq-2}) for $\pi$,
\begin{gather}\label{sec4-eq-6}
K_X+\Delta=\pi^{\ast}K_Y+p\Delta,
\end{gather}
Let now $U\subset X$ be the smooth part of $X$ and $V=\pi(U) \subset Y$. Then $V$ is also open. Since $\pi$ is purely inseparable of degree $p$, if $L$ is an invertible sheaf on $U$, then $L^p=\pi^{\ast}N$, where $N$ is an invertible sheaf on $V$~\cite[Proposition 3.8]{Tz17b}. Therefore,
\[
(\mathcal{O}_X(\Delta)|_U)^p=\pi^{\ast}M_V,
\]
where $M_V$ is an invertible sheaf on $V$. Since $X$ and $Y$ are normal, $U$ and $V$ have codimension 2 in $X$ and $Y$, respectively, and therefore it easily follows that 
\[
\mathcal{O}_X(p\Delta) = (\pi^{\ast}M)^{[1]},
\]
where $M=i_{\ast}M_V$, $i \colon V \rightarrow Y$ is the inclusion. From this and the equation (\ref{sec4-eq-6}) the claim follows. Therefore also 
\[
\mathcal{O}_X(3K_X+3\Delta) =(\pi^{\ast}N)^{[1]},
\]
where $N=M^{[3]}$. Hence by Proposition~\ref{sec1-prop1}, there exists a $k$-linear map
\[
D^{\ast} \colon H^0(\mathcal{O}_X(3K_X+3\Delta)) \rightarrow H^0(\mathcal{O}_X(3K_X+3\Delta)).
\]

Let $C \in |3K_X+3\Delta|$ be a curve which corresponds to an eigenvector of $D^{\ast}$. Then by Proposition~\ref{sec1-prop1}, $C$ is stabilized by $D$. Moreover, from the equations (\ref{sec4-eq-1000})  it follows that 
\begin{gather}\label{sec4-eq-7}
K_X\cdot C=3K_X^2+3K_X\cdot \Delta \leq 12 K_X^2,\\
C^2 =9K_X^2+9\Delta^2+18 K_X \cdot \Delta\leq 144K_X^2 \nonumber
\end{gather}

Let now $C=\sum_{i=1}^s n_i C_i$ be the decomposition of $C$ into its prime divisors. The assumptions of the proposition imply that $K_X\cdot C=3K_X^2 <p$. Hence by Corollary~\ref{sec1-cor1}, every component $C_i$ of $C$ is stabilized by $D$ and hence $D$ induces a vector field on $C_i$, for all $i$. The induced vector field will be non zero if and only if $C_i$ is not contained in the divisorial part $\Delta$ of $D$.

\textbf{Claim:} Suppose that $C_i$ is not contained in the divisorial part $\Delta$ of $D$. The $C_i$ a rational curve.

I proceed now to prove the claim.  Let $C_i^{\prime}=f_{\ast}^{-1}C_i$ be the birational transform of $C_i$ in the minimal resolution  $X^{\prime}$ of $X$. Then  $C_i^{\prime}$ is stabilized by $D^{\prime}$. Let $\nu \colon \bar{C}_i \rightarrow C^{\prime}_i$ be the normalization of $C^{\prime}_i$, which is also the normalization of $C_i$.Since $K_X$ is ample, it follows from the equations (\ref{sec4-eq-7}) than $K_X \cdot C_i \leq K_X\cdot C \leq 12K_X^2$. Then also
\[
K_X^{\prime} \cdot C_i^{\prime}=f^{\ast}K_X\cdot C_i^{\prime}=K_X\cdot C_i \leq 12K_X^2,
\]
and from the Hodge Index Theorem, 
\[
(C^{\prime}_i)^2\leq  \frac{(K_{X^{\prime}}\cdot C_i^{\prime})^2}{K_{X^{\prime}}^2}\leq 144 K_X^2.
\]
Therefore from the adjunction formula, \[
\mathrm{p}_a(C_i^{\prime})\leq 78K_X^2+1<(p-1)/2,
\]
from the assumptions of the proposition.  Hence from Proposition~\ref{sec1-prop4} it follows that $D^{\prime}$ fixes every singular point of $C^{\prime}_i$, for all $i=1,\ldots, s$ and $D^{\prime}$ lifts to a vector field $\bar{D}$ in the normalization $\bar{C}_i$ of $C^{\prime}_i$. Therefore $\bar{C}_i$ is either a smooth rational curve or an elliptic curve. I will show that $D^{\prime}$ has fixed points on $C_i^{\prime}$ and hence by 
Lemma~\ref{sec11-lemma-1}, $\bar{D}$ has also fixed points and hence $\bar{C}_i\cong \mathbb{P}^1$. Therefore $C_i$ is rational.

Next I will show that  there exists a fixed point of $D$ on $C_i$. Suppose that this was not the case and that $D$ has no fixed points on $C_i$. Then $C_i$ is in the smooth part of $X$ since by Theorem~\ref{sec11-th-2}, $D$ fixes every singular point of $X$. Then if $\tilde{C}_i=\pi(C_i)$, $\tilde{C}_i$ is in the smooth locus of $Y$. Since there are no fixed points of $D$ on $C$, $C \cdot \Delta =0$. Then from the adjunction formula for $\pi$ we get that 
\[
K_X\cdot C_i=\pi^{\ast}K_Y\cdot C_i=\pi^{\ast}K_Y\cdot \pi^{\ast}\tilde{C_i}=pK_Y\cdot \tilde{C_i},
\]
and therefore $K_X\cdot C_i \geq p$. On the otherhand it has been shown that $K_X\cdot C_i \leq 12 K_X^2<p$, by the assumptions of the proposition.   Hence there exists fixed points of $D$ on every $C_i$ and therefore $\bar{C}_i\cong \mathbb{P}^1$ and hence $C_i$ is rational as claimed.

Now let $\Delta^{\prime}=\sum_{i=1}^{\nu}n_i C_i$, where $C_i$, $1\leq i \leq \nu \leq s$ are the irreducible components of $C$ that are also components of $\Delta$ (and hence the restriction of $D$ on $C_i$ is zero). Let also $Z=\sum_{j=\nu +1}^{s}n_i C_i$, where $C_j$ are the irreducible components of $C$ which are not contained in $\Delta$ and therefore the restriction of $D$ on $C_j$, $j \geq \nu +1$, is not zero (if $\nu =s$ then $Z=0$). Then $C=\Delta^{\prime} +Z$. 

Next I will show that $Z\not= 0$ and that there is a component of it which dominates $B$. Hence $B$ is rational. 

Suppose that this is not true and that either $Z=0$ or no component of $Z$ dominates $B$. Therefore either $Z=0$ or $Z$ is contained in a finite union of fibers of $\psi h\colon X\rightarrow B$. Let $F$ be a general fiber of $\psi h$. Then in both cases $F \cdot Z=0$. Then if we write $3K_X=C-3\Delta=\Delta^{\prime}+Z-3\Delta$, the adjunction formula for $\pi$ becomes
\[
\Delta^{\prime}+Z=3\pi^{\ast}K_Y +3p\Delta.
\]
Intersecting this with a general fiber $F$ and taking into consideration that $F\cdot Z=0$ and that $F\cdot \pi^{\ast}K_Y=-2p$ we find that
\begin{gather}\label{sec4-eq-8}
\Delta^{\prime}\cdot F=-6p+3p( \Delta \cdot F).
\end{gather}
Now 
\begin{gather}\label{sec4-eq-9}
\Delta^{\prime}\cdot F=\sum_{i=1}^{\nu} n_i( C_i \cdot F)\leq m \left(\sum_{i=1}^{\nu}(C_i \cdot F)\right) \leq m \Delta \cdot F,
\end{gather}
where $m$ is the maximum among the $n_1, \ldots ,n_{\nu}$ such that $C_i \cdot F \not= 0$. Notice that it is not possible that $C_i \cdot F=0$, for all $i=1,\ldots, \nu$. If this was the case, then $\Delta^{\prime}\cdot F=0$. But since also we assume that $Z\cdot F=0$, it would follow that $C\cdot F=0$ and hence $(K_X+\Delta)\cdot F=0$. But then
\[
K_X \cdot F=-\Delta \cdot F \leq 0,
\]
for a general fiber $F$. But this is impossible since $K_X$ is ample. Hence $\Delta^{\prime}\cdot F>0$ and hence $m>0$.

Next I will show that $m \leq 12K_X^2$. Indeed. From the definition of $\Delta^{\prime}$ and the equation (\ref{sec4-eq-7}) it follows that
\[
m \leq \sum_{i=1}^{\nu}n_i\leq \sum_{i=1}^{\nu} n_i (K_X \cdot C_i) =K_X \cdot \Delta^{\prime} \leq K_X \cdot C \leq 12K_X^2,
\]
as claimed. Then from the equations (\ref{sec4-eq-8}), (\ref{sec4-eq-9}) it follows that
\begin{gather}\label{sec4-eq-10}
(12K_X^2-3p)\Delta\cdot F +6p >0.
\end{gather}
Notice now that from the adjunction formula for $\pi$ it follows that
\[
K_X\cdot F=\pi^{\ast}K_Y\cdot F+(p-1)\Delta \cdot F=-2p+(p-1)\Delta\cdot F.
\]
Then since $K_X \cdot F>0$, it follows that $\Delta \cdot F \geq 3$. Now the assumptions of the proposition imply that $K_X^2<p/12$. Then it is easy to see that
\[
(3p-12K_X^2)\Delta \cdot F-6p >0,
\]
which is a contradiction to the equation (\ref{sec4-eq-10}). Therefore it is not possible that $Z\cdot F=0$. Hence there exists a component $C_i$ of $C$ such the restriction of $D$ on $C$ is not zero and $C_i$ dominates $B$. Then since $C_i$ is rational, it follows that $B\cong \mathbb{P}^1$. This concludes the proof of Proposition~\ref{sec6-prop}.

 \end{proof}


\end{document}